\documentclass[12pt]{amsart}

\usepackage{amsmath, amssymb, amsthm, latexsym, amscd}
\usepackage[all, knot]{xy}
\xyoption{arc}
\usepackage{graphicx}

 \newlength{\baseunit}               
 \newcount{\numlines}                
\numberwithin{equation}{section}
\newtheorem{theorem}[equation]{Theorem}
\newtheorem{prop}[equation]{Proposition}
\newtheorem{lm}[equation]{Lemma}
\newtheorem{cor}[equation]{Corollary}

\theoremstyle{definition}
\newtheorem{definition}[equation]{Definition}

\newtheorem{example}[equation]{Example}
\newtheorem{examples}[equation]{Examples}
\newtheorem{remark}[equation]{Remark}
\newtheorem{remarks}[equation]{Remarks}

\DeclareMathOperator{\coker}{coker}
\DeclareMathOperator{\Hom}{Hom}
\DeclareMathOperator{\Tors}{Tors}
\DeclareMathOperator{\Dim}{Dim}
\DeclareMathOperator{\Det}{Det}
\DeclareMathOperator{\dom}{dom}
\DeclareMathOperator{\cod}{cod}
\DeclareMathOperator{\Ob}{Ob}
\DeclareMathOperator{\Aut}{Aut}
\DeclareMathOperator{\Ind}{Ind}
\DeclareMathOperator{\Pro}{Pro}
\DeclareMathOperator{\Mult}{Mult}

\DeclareMathOperator{\Fun}{Fun}

\DeclareMathOperator{\Pic}{Pic}

\newcommand{\Zee} {\mathbb{Z}}
\newcommand{\g} {\mathfrak{g}}
\newcommand{\h} {\mathfrak{h}}

\newcommand{\vect}{{\bf Vect}_0(k)}
\newcommand{\onevect}{{\bf Vect}_1(k)}
\newcommand{\allvect}{{\bf Vect}(k)}
\newcommand{\rar}{\rightarrow}

\newcommand{\lrar}{\longrightarrow}

\newcommand{\hrar}{\hookrightarrow}
\newcommand{\hlar}{\hookleftarrow}
\newcommand{\epi}{\twoheadrightarrow}

\newcommand{\limi}{\underset{i\in I}{``\varinjlim"}}

\newcommand{\limj}{\underset{j}{\varinjlim}}

\newcommand{\limproi}{\underset{i\in I}{``\varprojlim"}}
\newcommand{\limproj}{\underset{j\in J}{``\varprojlim"}}

\newcommand{\proYj}{\underset{j\in J}{``\varprojlim"} Y_j}
\newcommand{\proYij}{\underset{j\in J}{``\varprojlim"} Y_{i,j}}

\newcommand{\proXj}{\underset{j\in J}{``\varprojlim"}X_j}
\newcommand{\proXij}{\underset{j\in J}{``\varprojlim"}X_{i,j}}
\newcommand{\proZj}{\underset{j\in J}{``\varprojlim"}Z_j}
\newcommand{\proUj}{\underset{j\in J}{``\varprojlim"}U_j}
\newcommand{\proVj}{\underset{j\in J}{``\varprojlim"}V_j}

\newcommand{\limiriga}{\underset{i\in I}{``\varinjlim"}X_i}
\newcommand{\limjrigaX}{\underset{j\in J}{``\varinjlim"}X_j}
\newcommand{\limjrigaY}{\underset{j\in J}{``\varinjlim"}Y_j}

\newcommand{\indjriga}{\underset{j\in J}{``\varinjlim"}X_j}

\newcommand{\limprigaXi}{\underset{i\in I}{``\varprojlim"} X_i}

\newcommand{\limprigaY}{\underset{j\in J}{``\varprojlim"} Y_j}

\newcommand{\indproriga}{\underset{j}{``\varinjlim"}``\underset{i}{\varprojlim"}X_{i,j}}

\newcommand{\limA}{\underleftrightarrow{\lim \ }\mathcal A}

\newcommand{\dlim}{\underleftrightarrow{\lim}}
\newcommand{\xrar}{\xrightarrow}

\newcommand{\A}{\mathcal A}

\newcommand{\D}{\mathcal D}

\newcommand{\E}{\mathcal E}
\newcommand{\F}{\mathcal F}
\newcommand{\Pp}{\mathcal P}

\newcommand{\T}{\mathcal T}
\newcommand{\Ll}{\mathcal L}

\newcommand{\pt}{\partial}
\newcommand{\simpl}{\Sigma_{\bullet}}
\newcommand{\SA}{S_{\bullet}(\A)}
\newcommand{\SSA}{S_{\bullet}S_{\bullet}(\A)}

\begin{document}
\pagestyle{plain} 
\title{Sato Grassmannians for generalized Tate spaces}
\author{Luigi Previdi}
\date{}
\maketitle
\begin{abstract}
We generalize the concept of Sato Grassmannians of  locally linearly compact topological vector spaces (Tate spaces) to the category $\limA$ of the ``locally compact objects" of an exact category $\A$, and study some of their properties. This allows us to generalize the Kapranov dimensional torsor $\Dim(X)$ and  determinantal gerbe $\Det(X)$ for the objects of $\limA$ and unify their treatment in the {\it determinantal torsor} $\D(X)$. We then introduce a class of exact categories, that we call {\it partially abelian exact}, and prove that if $\A$ is partially abelian exact, $\Dim(X)$ and $\Det(X)$ are  multiplicative in admissible short exact sequences. When $\A=\vect$, the category of finite dimensional vector spaces on a field $k$, we recover the case of the dimensional torsor and of the determinantal gerbe of a Tate space, as defined by Kapranov in \cite{ka}, and reformulate its properties in terms of the Waldhausen space $S(\A)$ of the exact category $\A$. The advantage of this approach is that it allows to define formally in the same way the Grassmannians of the iterated categories $\dlim^n\A$. We then prove that the category of Tate spaces $\T=\dlim\vect$ is partially abelian exact, which allows us to extend the results on $\Dim$ and $\Det$ already known for Tate spaces to 2-Tate spaces, such as the multiplicativity of $\Dim$ and $\Det$ for 2-Tate spaces, as considered by Arkhipov-Kremnizer and Frenkel-Zhu. 
\end{abstract} 

\tableofcontents
\setcounter{tocdepth}{1}

\section{Introduction}

Let $k$ be a field, and consider the Tate space $V=k((t))$. For such a space $V$, the group $GL(V)$ (sometimes called the ``Japanese group" $GL(\infty)$) has properties which are quite different from those of the naively defined group $GL_{\infty}=\bigcup GL(n)$. In particular, it is typically disconnected, with $\pi_0(GL(V))=\Zee$. This has been interpreted by M. Kapranov in \cite{ka} in terms of the {\it dimensional torsor} $\Dim(V)$, naturally associated with $V$, which gives rise to a class in $H^1(GL(V), \Zee)=\Hom(GL(V), \Zee)$. Kapranov also proves that, for all Tate spaces $V$, the dimensional torsor $\Dim(V)$ is multiplicative with respect to admissible short exact sequences of Tate spaces.

\vspace{0.3cm}

In the language of exact categories,  Kapranov's results amount to  the consideration of the dimensional torsor $\Dim(V)$ for the objects $V$ of the exact category of Tate spaces $\T=\dlim\vect$ (see sect. \eqref{tate}), where $\vect$ is the category of finite 
dimensional vector  spaces over the field $k$. In this paper we generalize the Kapranov dimensional torsor to the Beilinson category 
$\limA$, where $\A$ is an  exact category. Objects of $\limA$ will serve as categorical generalizations of Tate spaces referred to in 
the title of this article. We prove the multiplicativity of $\Dim(V)$, and sketch the analog theory for the determinantal gerbe $\Det(V)
$, under the extra assumption that $\A$ has
pullbacks of admissible monomorphisms and pushouts of admissible epimorphisms. We call such categories ``partially abelian exact", 
since they can equivalently be described as exact categories such that, for any morphism $f$ which is the composition of an 
admissible monomorphism followed by an admissibe epimorphism, $f$ can be written in a unique way (up to isomorphisms) as the 
composition of an admissible epimorphism followed by an admissible monomorphism. For example, we prove that the category $\T$ 
of Tate spaces is partially abelian exact, and thus our theory applies to the category $\T_2=\dlim(\T)=\dlim\dlim\vect$, whose objects 
can be called {\it 2-Tate spaces}. For example, for a field $k$, the space $k((t))((s))$ is a 2-Tate space over $k$. Study of 2-Tate 
spaces was recently taken up by  Arkhipov and Kremnizer in \cite{ak} and by Frenkel and Zhu in \cite{fz}, in connection with 
representations of double loop groups. In the same order of ideas, Gaitsgory and Kazhdan have recently provided a categorical framework for the study 
of the representations of the group $G({\mathbb F})$, where $G$ is reductive  and ${\mathbb F}$ is a 2-dimensional local field \cite{gk}. In a recent paper \cite{dr}, Drinfeld defined the notion of dimensional torsor in the more general situation of modules over a commutative ring, and defined the \'etale local notion of Tate module. Our results provide a categorical foundation for this study.

\vspace{0.3cm}

In order to generalize the dimensional torsor and the determinantal gerbe to the objects $X$ of the Beilinson category $\limA$, for 
$\A$ exact, we introduce an appropriate concept of Grassmannians for $\limA$, which generalizes the  {\it Sato Grassmannians}, 
originally defined by Sato in \cite{s} for the category of Tate spaces. Our definition  uses the language of ind/pro-objects on $\A$, 
which has the advantage to allow us to defined formally in the same way the Grassmannians for all the iterated categories 
$\dlim\dlim\A, \cdots, \dlim^n\A$. We then study the behavior of the Grassmannian of an object $X$ with respect to admissible 
short exact sequences of $\limA$, when $\A$ is partially abelian exact. This allows us to define the {\it determinantal torsor} 
$\D$ for the objects of the Beilinson category $\limA$. This is a torsor defined over a certain Picard category 
$\Pp$. When $\Pp=V(\A)$, the symmetric category of virtual objects on $\A$ defined by Deligne (cf. \cite{d}), the determinantal 
torsor $\D(X)$ encloses the datum of  the $K_0(\A)$-torsor $\Dim(X)$  and of the $K_1(\A)$-gerbe $\Det(X)$. In particular, for 
$\A=\vect$, it is $K_0(\A)=\Zee$ and $K_1(\A)=k^*$, and this construction provides a unified treatment of the Kapranov $\Zee$-dimensional torsor $\Dim(V)$ and $k^*$-gerbe $\Det(V)$, and extends the theory of \cite{ka} to the general $K$-theoretic setting.

\vspace{0.3cm} {\bf Acknowledgements.} This paper is the second part of the dissertation that I presented to the Faculty of the Graduate School of Yale University in partial fulfillment of the requirements for the Degree of Doctor of Philosophy in Mathematics. I would like to express my gratitude to my advisor, Mikhail Kapranov, for his constant assistance and help, and to Alexander Beilinson, who read a preliminary version of this work and made many remarks;  in particular, he pointed out to me the importance of the symmetry condition for determinantal theories, which appears to be crucial in the developement of the theory here proposed.

\section{Picard categories and the Waldhausen space}
\subsection{Generalities on exact categories}\label{quillen}
Let be $\A$ an exact category, in the sense of Quillen (\cite{q}). Recall that this means that $\A$ is endowed with a class $\E$ of sequences
$$
a'\stackrel{i}\hrar a \stackrel{j}\epi a''
$$
called {\it admissible short exact sequences}, which satisfy certain axioms (see \cite{q}). An admissible monomorphism is a 
morphism which appears as $i$ and an admissible epimorphism is a morphism which appears as $j$ in such a sequence. 

\vspace{0.2cm}

Equivalently, an exact category $\A$ can be described as a full subcategory of an abelian category $\F$ which is {\it closed under 
extensions}, i.e. whenever $a'\hrar x\epi a''$ is a short exact sequence of $\F$ with $a, a''\in\A$, we have $x\in\A$. Given an exact 
category $\A$, it is always possible to construct an embedding $h:\A\hrar\F$, where $\F$ is an abelian category, such that a 
sequence $a'\rar a\rar a''$  is an admissible short exact sequence of $\A$  if and only if $h$ carries it into a short exact sequence of 
$\F$. We then call $\F$ the {\it abelian envelope} of the exact category $\A$ and $h$ is called the {\it Quillen embedding}, see \cite{q}.

\vspace{0.1cm}

The following will be useful:

\begin{lm}\label{admono} (cf. \cite{pre}) Let be $f: a\epi b$ an epimorphism of $\F$, with $a, b\in\A$. Then $f$ is an admissible 
epimorphism of $\A$ if and only if $\ker(f)$ is in $\A$. Dually for  monomorphisms $g: c\hrar d$.
\end{lm}

\begin{lm}\label{pullback} (cf. \cite{pre})
A pullback diagram in the category $\A$ remains a pullback diagram in the category $\F$.
\end{lm}

\begin{definition} An {\it admissible subobject} of an object $a\in\A$ is a class of admissible monomorphisms $a'\hrar a$ modulo 
the equivalence relation given by  $(a'\hrar a) \sim (a''\hrar a)$ if and only if there exists an isomorphism $a'\xrar{\sim} a''$ such that
$$
\xymatrix{ 
a'  \ar@{^{(}->}[rd]  \ar[rr] &&a''\ar@{_{(}->}[ld] \\
& a  &
}
$$
is commutative.
\end{definition}
 
 As in \cite{pre}, we call a commutative square {\it admissible}
$$
\xymatrix{ 
X_{}\ar[d]
\ar[r]
& Y_{} \ar[d] \\
Z_{}\ar[r]_{}
& V_{}
}
$$
if the horizontal arrows are admissible monomorphisms and the vertical ones are admissible epimorphisms. If such a square is cartesian, it is also cocartesian, and vice versa (see again \cite{pre}).

\subsection{The Waldhausen S-construction}\label{SA}
In this section, we refer to \cite{gm} and \cite{w} for the terminology relative to simplicial categories. 

\vspace{0.1cm}

Given an exact category, Waldhausen \cite{w} associates to it a simplicial category $\SA$, whose geometric realization (as defined e.g. in \cite{gm} or in \cite{gz}) $S(\A)$ provides a topological model for the $K$-theory of $\A$, i.e., $K_i(\A)=\pi_{i+1}S(\A)$ (see \cite{w2}). 

\begin{definition}\label{defsa}
Let $\A$ be an exact category and $n\geq 0$ an integer. The category $S_n(\A)$ is defined as the category whose objects are data 
$\{\underline{a}\}$ consisting of:
\begin{itemize}
\item{
objects $a_{ij}\in\A$, given for each $(i, j)$ with $0\leq i\leq j\leq n$.
}
\item{
morphisms $\phi^{kl}_{ij}\colon a_{ij}\rar a_{kl}$, given for $i\leq k$, $j\leq l$ (we shall write $(i, j)\leq (k, l)$)
}
\end{itemize}
such that the following conditions are satisfied:
\begin{itemize}
\item{
For all $(i, j, k)$, with $i\leq j\leq k$,
$$
a_{ij}\xrar{\phi^{ik}_{ij}}a_{ik}\xrar{\phi^{jk}_{ik}}a_{jk}
$$
is an admissible short exact sequence of $\A$.
}
\item{
If $(i,j)\leq (k,l)\leq (m,n)$, we have a commutative diagram
$$
\phi^{mn}_{ij}=\phi^{mn}_{kl}\phi^{kl}_{ij}.
$$
}
\end{itemize}
A {\it morphism} between two objects $\underline{a}\lrar\underline{b}$ of $S_n(\A)$ is by definition a collection of isomorphisms $a_{ij}\stackrel{\sim}\rar b_{ij}$, $\forall i\leq j$, making the resulting diagrams commutative.
\end{definition}
In particular, $a_{ii}=0$ and we see that $\{\underline{a}\}$ gives rise to a {\it rigidified admissible
filtration} of objects of $\A$ of length $n$, i.e. a sequence of $n$ admissible monomorphisms of the following type:
$$
\underline{a}=0\hrar a_1\hrar a_2\hrar\cdots\hrar a_n
$$
toghether with a compatible choice of an object $a_{ij}$ in the isomorphism class of each quotient, so that $a_{ij}=\dfrac{a_j}{a_i}$, for $i\leq j$ and there is a commutative diagram
\begin{equation}\label{waldhausen}
\xymatrix{ 
a_{01}  \ar@{^{(}->}[r]& a_{02}\ar@{->>}[d]\ar@{^{(}->}[r] & a_{03}\ar@{->>}[d]\ar@{^{(}->}[r]&\cdots\ar@{^{(}->}[r] & a_{0n}\ar@{->>}[d]\\
& a_{12} \ar@{^{(}->}[r] & a_{13} \ar@{->>}[d]\ar@{^{(}->}[r]&\cdots\ar@{^{(}->}[r] & a_{1n}\ar@{->>}[d]\\
&& a_{23}\ar@{^{(}->}[r]&\cdots\ar@{^{(}->}[r] & a_{2n}\ar@{->>}[d]\\
&&                     &                    & \vdots\ar@{->>}[d] \\
&&                     &                    & a_{n-1n}
}
\end{equation}
whose horizontal arrows are admissible monomorphisms and the vertical arrows are admissible epimorphisms. 

\vspace{0.2cm}

For each $n\geq 0$, we define a functor $\pt_0: S_n(\A)\rar S_{n-1}(\A)$ by erasing the top row of \eqref{waldhausen} and reindexing. Then, $\pt_0(\underline{a})=a_{12}\hrar\cdots\hrar a_{1n}$, with $\pt_0(\underline{a})_{i,j}=a_{i+1, j+1}$; we define a functor $\pt_i:S_n(\A)\rar S_{n-1}(\A)$ for all $0<i\leq n$ by erasing the row $a_{i,*}$ and the column containing $a_i$. 

\vspace{0.2cm} The functors $s_i:S_n(\A)\rar S_{n+1}(\A)$, for $0\leq i\leq n$ are defined by doubling the object $a_i$ in $(\underline{a})$. Then, we have the following

\begin{prop} (\cite{w})
The system $(S_n(\A),\pt_i, s_j)$ forms a simplicial category $S_{\bullet}(\A)$.
\end{prop}
Next, the geometric realization of $S_{\bullet}(\A)$ is constructed as follows.  Since $\SA$ is a simplicial category, we consider the geometric realizations $|S_n(\A)|$ of the categories $S_n(\A)$. These form a 
simplicial topological space $BS_{\bullet}(\A)$; we then take the geometric realization of $BS_{\bullet}(\A)$, 
and call it $S(\A)$. Thus, $S(\A)=|S_{\bullet}(\A)|$. This space is called the {\it Waldhausen space} associated with the exact category $\A$. Notice that the simplicial space $BS_{\bullet}(\A)$ is a bisimplicial set, and the 
space $S(\A)$ can be interpreted as the geometric realization of this bisimplicial set.

\begin{remark}\label{bisimplexes}
The geometric realization $S(\A)$ is thus constructed out of the $(p, q)$-bisimplices $\Delta^p\times\Delta^q$ glued together along the face maps of the bisimplicial set $S_{\bullet}(\A)$. The bisimplices of dimension $\leq 3$ are parametrized as follows:
\begin{itemize}
\item{$\Delta^0\times\Delta^0$: only one point (basepoint) $*$  in $S(A)$.
}
\item{$\Delta^1\times\Delta^0$: one for each  object $\{a\}$ of $\A$; geometrically, this gives rise in $S(\A)$ to a loop (embedded circle) at $*$ which we denote by $|a|$.
}
\item{$\Delta^1\times\Delta^1$: one for each isomorphism $\{a\xrar{\sim} b\}$ of $\A$, giving rise to a homotopy between the loops $|a|\sim|b|$, hence to an element of $\pi_2(S(A),*)$.
}
\item{$\Delta^2\times\Delta^0$: one for each admissible short exact sequence $\{\sigma\colon a'\hrar a\epi a''\}$. Geometrically, 2-simplexes as in the picture
\begin{figure}[h]
\centerline{\includegraphics{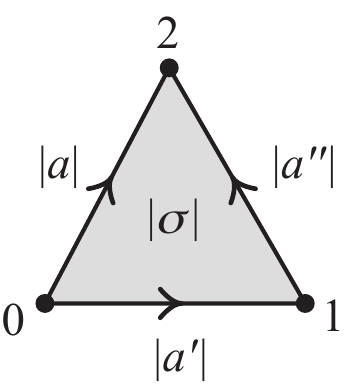}}
\caption{}\label{Fi:sigma}
\end{figure}
}
\item{$\Delta^2\times\Delta^1$: one for each isomorphism of admissible short exact sequences $\{\sigma_0\xrar{\sim}\sigma_1\colon\sigma_0, \sigma_1\in S_{2,0}(\A)\}$. Geometrically, the filled prism whose bottom is the 2-simplex $|\sigma_0|$ and whose top is the 2-simplex $|\sigma_1|$.
}
\item{$\Delta^1\times\Delta^2$: one for each composable pair of isomorphisms of $\A$: $\{a\xrar{\sim} b\xrar{\sim} c\}$
}
\item{$\Delta^3\times\Delta^0$: one for each rigidified admissible filtration of lenght 2 of $\A$ $\{\tau\colon a_1\hrar a_2\hrar a_3\}$. Geometrically, the filled tetrahedron generated by the $a_i$'s as in the figure

\begin{figure}[h]
\centerline{\includegraphics{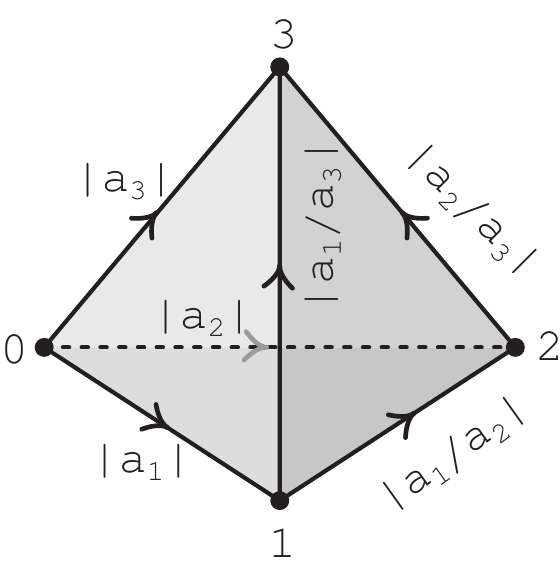}}
\caption{}\label{Fi:tau}
\end{figure}
}
\item{and so on.
}
\end{itemize}
In particular the  space $S(\A)$ is connected, since $|S_0(\A)|=*$. 
\end{remark}
\subsection{Iteration of the S-construction and delooping}\label{delooping}
 In \cite{w2} Waldhausen proves also that the space $S(\A)$ admits a delooping. Such  delooping is constructed as the geometric realization $SS(\A)$ of a bisimplicial category $\SSA$, obtained by ``iterating" the $S$-construction. Roughly speaking, the $(p, q)$-bisimplexes of $\SSA$ are $(p, q)$-rigidified admissible bifiltrations of objects of $\A$. By this expression we mean a commutative diagram
$$
\xymatrix{ 
a_{11} \ar@{^{(}->}[r]\ar@{^{(}->}[d]& a_{12}\ar@{^{(}->}[d]\ar@{^{(}->}[r] &\cdots\ar@{^{(}->}[r] & a_{1q}\ar@{^{(}->}[d]\\
\vdots&\vdots&\cdots & \vdots\\
\vdots\ar@{^{(}->}[d]&\vdots\ar@{^{(}->}[d]&\cdots & \vdots\ar@{^{(}->}[d]\\
a_{p1} \ar@{^{(}->}[r]& a_{p2}\ar@{^{(}->}[r] &\cdots\ar@{^{(}->}[r] & a_{pq}\\
}
$$
 such that each horizontal and vertical arrow is an admissible monomorphism, and rigidified similarly to Definition \eqref{defsa}.  We refer again to \cite{w2} for details. The clasifying space  $SS(\A)=|\SSA|$ is thus the geometric realization of a trisimplicial set. We get $S(\A)=\Omega SS(\A)$, and it is possible to furtherly iterate the $S$-construction to obtain an $n$-simplicial category $S^n_{\bullet}(\A)$, and prove that 
$S(\A)=\Omega^{n-1}(S^n(\A))$. As a corollary, we have that $S(\A)$ is an infinite loop space.

\vspace{0.1cm} 

Note that every object $\underline{a}$ of $S_n(\A)$ gives an object $\alpha(\underline a)$ of $S_nS_1(\A)$ and an object 
$\beta(\underline a)$ of $S_1S_n(\A)$ (bifiltrations going purely horizontally or purely vertically). We have therefore two maps of 
the suspension
$$
\Sigma S(\A)\rar SS(\A)
$$
both adjoint to the delooping isomorphism
$$
S(\A)\rar \Omega SS(\A).
$$
On the level of cells, each $(p, q)$-cell $\sigma$ of $S(\A)$ gives rise to a $(p, 1, q)$-cell $\alpha(\sigma)$ and to a $(1, p, q)$-cell 
$\beta(\sigma)$ of $SS(\A)$. Notice that up to dimension 4, all cells of $SS(\A)$ are obtained in this way except for the cells of the following type:
\begin{itemize}
\item{$\Delta^2\times\Delta^2\times\Delta^0$: one for each diagram of objects of $\A$:
\begin{equation}\label{symmetricintrinsic}
\xymatrix{ 
x^1_1 \ar@{^{(}->}[r]\ar@{^{(}->}[d] &x^1 \ar@{->>}[r]\ar@{^{(}->}[d] &x^1_2\ar@{^{(}->}[d]^{}\\
x_1 \ar@{^{(}->}[r]\ar@{->>}[d] &x\ar@{->>}[d]\ar@{->>}[r] &x_2\ar@{->>}[d]^{}  \\
x^2_1 \ar@{^{(}->}[r]_{} &x^2\ar@{->>}[r]_{} & x^2_2 \\
}
\end{equation}
whose rows and columns are admissible short exact sequences. 
}
\end{itemize}
\begin{remark} It is important to notice that  in the diagram \eqref{symmetricintrinsic} one has to impose the admissibility of the 
sequences of the quotients: for general exact categories $\A$ this condition does not descend from the admissibility of the 
monomorphisms which appear in the top left square.
\end{remark}
\subsection{Torsors over abelian groups}

\begin{definition} Let $G$ be a group (assumed to be abelian in the sequel). A {\it $G$-torsor} is a set $T$ with an action of $G$ 
which is free and transitive. If $T$ and $T'$ are $G$-torsors, a {\it morphism of G-torsors} is a map of the corresponding $G$-sets 
$f:T\lrar T'$.
\end{definition}

\begin{remark} It follows immediately from the definition that an isomorphism of $G$-torsors is a morphism which is a bijection of the underlying sets, and also that every morphism of $G$-torsors is in fact an isomorphism.
\end{remark}
Given two $G$-sets $S$ and $T$, with $G$ abelian, it is possible to define thier {\it tensor product} $S\otimes T$ as the set $S\times T$ quotiented by the equivalence relation generated by 
$$
(gs, t)\sim (s,gt).
$$
It is straightforward to check that the resulting set $S\otimes T$ is naturally a $G$-set. If $S$ and $T$ are both $G$-torsors, then $S\otimes T$ is a $G$-torsor. It is clear that we have a canonical isomorphism 
$S\otimes T\xrar{\sim} T\otimes S$.

\vspace{0.1cm}

Let be $\phi:G\rar H$ a morphism of abelian groups. Then $H$ is naturally a $G$-set. Let $T$ be a $G$-torsor. We define the {\it pushout} of $T$ along $\phi$ to be the $H$-set
$$
\phi_*T:=H\otimes_G T
$$
The following is easily proved:

\begin{prop}
The $H$-set $\phi_*T$ is naturally an $H$-torsor.
\end{prop}

\subsection{Picard categories}

We recall some welll known facts from \cite{d}.

\vspace{0.1cm}
Let be ($\Pp, \otimes, \alpha$) a category with an associative tensor product $\otimes$, i.e. for all objects $x_1, x_2, x_3$ there are 
natural isomorphisms $\alpha=\alpha_{x_1,x_2,x_3}\colon x_1\otimes (x_2\otimes x_3)\rar (x_1\otimes x_2)\otimes x_3$ which 
satisfy the MacLane pentagonal diagram (cf. \cite{ml}). We say that ($\Pp, \otimes, \alpha$)  is a {\it Picard category} if $\Pp$ is such 
that every morphism is an isomorphism and for all object $x$ of $\Pp$, the functors $x\rar x\otimes p$ and $x\rar p\otimes x$ are 
self-equivalences. It follows that $\Pp$ has a unit object, {\bf 1}, unique up to a unique isomorphism, and that each $x\in\Pp$ has a 
{\it  dual object} $x^*$ (also written $x^{-1}$), unique up to a unique isomorphism. A {\it symmetric} Picard category is a Picard category which is endowed with a symmetry $\sigma_{x,y}:x\otimes y\rar y\otimes x$ for each pair of objects, making it into a symmetric monoidal category 
in the  sense of MacLane (see \cite{ml}). A Picard category $\Pp$ is called {\it strictly symmetric} if $\sigma_{x,x}=1_{x\otimes x}$ for all $x$ in $\Pp$.
\vspace{0.1cm}

Let be $\Pp$ a Picard category. Introduce on $\Ob(\Pp)$  an equivalence relation by letting $x\sim y \Leftrightarrow$ there is an isomorphism $x\rar y$. The quotient set is denoted by $\pi_0(\Pp)$ and it is a group under the tensor product of $\Pp$ as the group multiplication of $\pi_0(\Pp)$. For a 
symmetric Picard category this group is abelian.
A Picard category is called {\it connected} if for any pair of objects $x, y$ there exists an isomorphism $x\rar y$, or, equivalently, if $\pi_0(\Pp)=0$.

\vspace{0.1cm} 

The group $\Aut_{\Pp}({\bf 1})$ is denoted by $\pi_1(\Pp)$. An application of the standard Eckmann-Hilton argument shows that this group is abelian. It is clear that $\pi_1(\Pp)$ is isomorphic (not canonically in general) to $\Aut_{\Pp}(x)$ for all objects $x\in\Pp$.
\begin{examples}\label{virtual}
\par
\begin{enumerate}
\item{Let be $G$ an abelian group. The category $\Tors(G)$ of the torsors over $G$ is a connected Picard category, with the monoidal structure given by the tensor product of torsors. The Picard category $\Tors(G)$ is strictly symmetric. The dual of a $G$-torsor $T$ is the $G$-torsor $\Hom(T,G)$. It is clear that $\Tors(G)$ is a connected Picard category, so $\pi_0(\Tors(G))=0$ and $\pi_1(\Tors(G))=G$.}

\vspace{0.1cm}

\item{ {\it The category of virtual objects of an exact category.} Following Deligne (\cite{d}), we associate to each exact category $\A$ a symmetric Picard category $V(\A)$, called the {\it category of virtual objects} of $\A$. Here is a slightly modified version of Deligne's construction: \\
- an object of $V(\A)$ is a loop of $S(\A) \gamma\colon [0, 1]\mapsto S(\A)$ (cf. sect. \eqref{SA}), with $\gamma(0)=\gamma(1)=*$. \\
- a morphism $\gamma_1\rar\gamma_2$ is a homotopy class rel $*$ of homotopies  from $\gamma_1$ to $\gamma_2$.

\vspace{0.1cm}

The composition of two morphisms $\gamma_1\xrar{[F]}\gamma_2\xrar{[G]}\gamma_3$ is defined as the class of the homotopy $F*G:\gamma_1\rar\gamma_3$. Since $F*(G*H)\sim (F*G)*H$, the composition of arrows is associative and $V(\A)$ is a category. The category $V(\A)$ is a Picard category, with the tensor product on objects $\gamma_1\otimes\gamma_2$ defined as the composition of loops $\gamma_1*\gamma_2$. The associativity constraint $\alpha$ is given by the class of  the standard homotopy of loops $\gamma_1*(\gamma_2*\gamma_3)\sim(\gamma_1*\gamma_2)*\gamma_3$. The unit object is the constant loop at 0. 
Further, $V(\A)$ 
admits a symmetry which makes it into a symmetric Picard category. To see 
this, consider the direct sum $\oplus$ in the exact category $\A$. The 
operation $\oplus$ makes $S(\A)$ into an $H$-space, whose sum will be 
still denoted by $\oplus$, commutative up to all higher homotopies. This 
defines a commutativity constraint on $V(\A)$, via:
$$
\gamma_1*\gamma_2\sim\gamma_1\oplus\gamma_2\sim\gamma_2\oplus
\gamma_1\sim\gamma_2*\gamma_1,
$$
cf. \cite{d}, 4.2.2. It is not difficult to see that the above ismorphism of 
$V(\A)$ makes $V(\A)$ into a symmetric Picard category, with 
$\pi_0(V(\A))=K_0(\A)$ and $\pi_1(V(\A))=K_1(\A)$. \\
In general, $V(\A)$ is not strictly symmetric.
}
\item{Let $\A=\vect$. In this case, the symmetric Picard category $V(\A)$ is
equivalent to the category $\Pic_k^{\Zee}$ of $\Zee$-graded 1-dimensional 
vector  spaces over $k$. This is a symmetric Picard category, with symmetry given 
as follows. Suppose that $L$ has degree $a$ and $M$ has degree $b$. Then, 
for all $x\in L$ and $y\in M$, we define 
$$
\sigma_{x,y}\colon x\otimes y\to (-1)^{ab}y\otimes x.
$$
The equivalence $V(\A)\sim \Pic_k^{\Zee}$ is mentioned in \cite{dr} (5.5.1).
}
\end{enumerate}
\end{examples}
 
\subsection{Torsors over a Picard category} We recall here a generalization of the concept of a torsor to Picard categories, which is discussed in a more general setting by Drinfeld in \cite{dr}.
\begin{definition}\label{pictorsor}
Let be $(\Pp, \otimes, \alpha, \sigma, {\bf 1})$ a symmetric Picard category. A {\it torsor} over $\Pp$ is a groupoid $\T$, together with a 
bifunctor 
$\otimes$:
\begin{align*}
\Pp\times\T&\xrar{\otimes}\T \\
(x,a) &\xrar{} x\otimes a,
\end{align*} for which:
\begin{itemize}
\item{There are natural isomorphisms $\beta_{x_1,x_2,a}:x_1\otimes(x_2\otimes a)\rar (x_1\otimes x_2)\otimes a$ for which the following diagram is commutative:
\[
\xymatrix@C=5em{
    x_1\otimes (x_2\otimes(x_3\otimes a))\ar[d]_{\beta_{x_1, x_2, x_3\otimes a}}\ar[r]^{1_{x_1}\otimes\beta_{x_2,x_3,a}}& x_1\otimes ((x_2\otimes x_3)\otimes a)\ar[r]^{\beta_{x_1, x_2\otimes x_3,a}}&(x_1\otimes (x_2\otimes x_3))\otimes a\ar[d]^{\alpha_{x_1,x_2,x_3\otimes 1_a}}\\
    (x_1\otimes x_2)\otimes (x_3\otimes a)\ar[rr]_{\beta_{x_1\otimes x_2,x_3,a}}&&((x_1\otimes x_2)\otimes x_3)\otimes a
}
\]}
\item{For all objects $a\in\T$ there is a natural isomorphism:
\begin{equation}\label{null}
\lambda_a\colon {\bf 1}\otimes a\xrar{\sim}a
\end{equation}
compatible with the associativity constraint $\alpha$ and with the isomorphism $\beta$.
}
\item{For all objects $a\in\T$, the induced functor:
\begin{equation}\label{freetransitive}
\Pp\xrar{-\otimes a}\T
\end{equation}
is an equivalence of categories.}
\end{itemize}
\end{definition}

Let $\T$ be a $\Pp$-torsor. In $\Ob(\T)$ we introduce an equivalence relation by letting $x\sim y \Leftrightarrow$ there is an isomorphism $x\rar y$. We denote by $[x]$ the equivalence class of the object $x$. We let: 
$$
\pi_0(\T):=\dfrac{\Ob(\T)}{\sim}.
$$
\begin{prop}\label{inducedtorsor}
Let be $\T$ a $\Pp$-torsor. Then there is an action of the group $\pi_0(\Pp)$ on the set $\pi_0(\T)$, which makes $\pi_0(\T)$ into a 
torsor over the abelian group $\pi_0(\Pp)$.
\end{prop}

\begin{proof}
The action is defined as 
$$
\begin{array}{ll}
\pi_0(\Pp)\times\pi_0(\T)\rar\pi_0(\T) \\
([g], [x])\mapsto [g\otimes x].
\end{array}
$$
The action is well defined and associative because of the conditions on the isomorphism $\beta$. The conditions on the isomorphism $\lambda$ gives $[1][x]=[x]$ for all object $x$. The action is free and transitive as a consequence of the condition \eqref{freetransitive}. 
\end{proof}

\subsection{Determinantal theories on exact categories with values on Picard categories}\label{detspace}

\begin{definition}\label{dettheory}
Let $\A$ be an exact category and $\Pp$ a symmetric Picard category. A {\it $\Pp$-valued determinantal theory} on $\A$ is a pair 
$(h, \lambda)$, where:
\begin{enumerate}
\item{$h:\Ob(\A)\lrar\Ob\Pp$ is a function such that $h(0)=\bf 1$.}
\item{For all isomorphisms $a\xrar{\sim} a'$, an isomorphism $h(a)\xrar{\sim} h(a')$, making $h$ into a functor, for which:
}
\item{$\lambda$ is a system of isomorphisms given for all admissible short exact sequences $\sigma=a'\hrar a\epi a''$ of $\A$:
$$
\lambda_{\sigma}\colon h(a')\otimes h(a'')\xrar{\sim}h(a),
$$
which are natural with respect to isomorphisms of admissible short exact sequences.} \\
These data are required to satisfy: \\
\item{For all admissible filtration of length 2 of $\A$, $a_1\hrar a_2\hrar a_3$ with a compatible choice of quotients, there is a 
commutative  diagram
\begin{equation}\label{mult}
\xymatrix{ 
h(a_1)\otimes h\left (\dfrac{a_{2}}{a_{1}}\right )\otimes h\left (\dfrac{a_{3}}{a_{2}}\right )\ar[d]_{\lambda\otimes 1} \ar[rr]^-{1\otimes\lambda} &&h(a_1)\otimes h\left (\dfrac{a_{3}}{a_{1}}\right )\ar[d]^{\lambda}\\
h(a_2)\otimes h\left (\dfrac{a_{3}}{a_{2}}\right ) \ar[rr]_-{\lambda}&& h(a_3) \\
}
\end{equation}
}
\end{enumerate}
(where we have omitted for simplicity the associator of $\Pp$.)
\end{definition}

A {\it morphism} of determinantal theories $(h, \lambda)\lrar (h', \lambda ')$ is a collection of morphisms 
$\{f_i\colon h(a_i)\to h '(a_i)\}$ of $\Pp$, such that, for all admissible short exact sequences $a'\hrar a\epi a''$, the diagram
$$
\xymatrix{
h(a')\otimes h(a'')\ar[r]^-{\lambda}\ar[d]_{f_{a'}\otimes f_{a''}} &h(a)\ar[d]^{f_a} \\
 h'(a')\otimes h'(a'')\ar[r]^-{\lambda'} & h'(a)
}
$$
commutes.

\vspace{0.1cm}

It is clear that every morphism of determinantal theories is an isomorphism. 
\begin{remarks} 
\begin{enumerate}
\item{Notice that from the functoriality of $h$ it follows that if $f:a\xrar{\sim}b$ is an isomorphism, and $\sigma\colon a\xrar{\sim} b\epi 0$ (resp. 
$\sigma\colon 0\hrar a\xrar{\sim} b$), one has $\lambda_{\sigma}=h(f):h(a)=h(a)\otimes h(0)\xrar{\sim} h(b)$
}
\item{The conditions defining a determinantal theory on $\A$ can be interpreted as conditions that $h$ must satisfy on the simplices of dimension $\leq 3$ of the simplicial Waldhausen category $\SA$. Indeed, notice in the first place that $h$ is a functor $S_1(\A)\rar\Pp$. Next, referring to the description of low-dimensional bisimplexes  given in section \eqref{SA}, $h$ is completely determined as a map which sends:
\begin{itemize}
\item{bisimplexes of type  $\Delta^0\times\Delta^p(=* \text{ in } S(\A))\lrar$ the null object $\bf 1$
}
\item{bisimplexes of type  $\Delta^1\times\Delta^0\lrar$ objects of $\Pp$
}
\item{bisimplexes of type  $\Delta^1\times\Delta^1\lrar$ isomorphisms of $\Pp$
}
\item{bisimplexes of type  $\Delta^1\times\Delta^2\lrar$ compositions of isomorphisms of $\Pp$
}
\item{bisimplexes of type  $\Delta^2\times\Delta^0\lrar$ isomorphisms of type $\lambda_{\sigma}$ of $\Pp$
}
\item{bisimplexes of type  $\Delta^2\times\Delta^2\lrar$ diagrams expressing the naturality of $\lambda_{\sigma}$ in $\Pp$
}
\item{bisimplexes of type  $\Delta^3\times\Delta^0\lrar$ commutative diagrams of type \eqref{mult}.
}
\end{itemize} 
}
\end{enumerate}
\end{remarks}

\begin{definition} Let be $\A$ an exact category and $\Pp$ a Picard symmetric category. The category (groupoid) $\Det(\A,\Pp)$ whose objects are the $\Pp$-valued determinantal theories on $\A$  and morphisms the isomorphisms of determinantal theories is called the category of {\it $\Pp$-valued determinantal theories on $\A$}, and denoted by $\Det(\A,\Pp)$.
\end{definition}

\subsection{Symmetric vs. non-symmetric determinantal theories}
We introduce now the ``symmetric versions" of the notions of determinantal theory and of $\Det(\A, \Pp)$, which will be central in the developement of our theory, as follows:

\begin{definition}\label{newsymmetric}
Let be $\Pp$ a symmetric Picard category, with symmetry $\sigma$. A 
{\it $\Pp$-valued symmetric determinantal theory} on $\A$  is a $\Pp$-valued determinantal theory $(h, \lambda)$ on $\A$, such that, for all diagrams of type  \eqref{symmetricintrinsic}, the following diagram
\begin{equation}\label{detsymmetric}
\xymatrix{  
h(x^1_1)\otimes h(x^1_2)\otimes h(x^2_1)\otimes h(x^2_2)  \ar[d]_{\lambda\otimes\lambda}\ar[rr]^{1\otimes\sigma\otimes 1}   && h(x^1_1)\otimes h(x^2_1)\otimes h(x^1_2)\otimes h(x^2_2)\ar[d]^{\lambda\otimes\lambda} \\
h(x^1)\otimes h(x^2)  \ar[rd]_{\lambda}   && h(x_1)\otimes h(x_2)\ar[ld]^{\lambda} \\
& h(x)  &
}
\end{equation}
is commutative.
\end{definition}
A morphism of symmetric determinantal theories is defined as in the general case.
\begin{definition} 
If $\Pp$ is a symmetric Picard category, we define $\Det_{\sigma}(\A,\Pp)$  to  be the  groupoid whose objects are the symmetric $\Pp$-valued   determinantal  theories on $\A$, and whose morphisms are the morphisms of determinantal theories.
\end{definition}
\begin{remark}
Thus, the datum of a symmetric determinantal theory is equivalent to a collection of data on the cells of $SS(\A)$  up to dimension $\leq 4$, which $(h, \lambda)$ must satisfy. Indeed, all such cells come from those of $S(\A)$, except for those of type 
$\Delta^2\times\Delta^2\times\Delta^0$, for which we impose the additional condition \eqref{detsymmetric}. Notice also that \eqref{detsymmetric} implies \eqref{mult}, when in \eqref{symmetricintrinsic} we let the left column to be the admissible short exact  sequence $x^1_1=x_1\epi 0$.
\end{remark}
The next proposition, which is a reformulation of a result due to Deligne (cf. \cite{d}, (4.8)) will be useful to perform computations.
\begin{prop}\label{symmdet}
Let $(h, \lambda)$ be a determinantal theory on $\A$, with values in the symmetric Picard category $\Pp$. Then $(h, \lambda)$ is symmetric if and only if for each pair of objects $a, b$ of $\A$, the diagram 
\begin{equation}\label{oldsymm}
\xymatrix{
h(a)\otimes h(b)\ar[drr]_{\sigma}\ar[rr]^{\lambda}&& h(a\oplus b)=h(b\oplus a)\ar[d]^{\lambda^{-1}} \\
&& h(b)\otimes h(a)
}
\end{equation}
commutes.
\end{prop}
\begin{proof}
We add the details to the argument sketched by Deligne. Since $\A$ is an exact category, it is closed under pushouts of admissible monomorphisms. Hence the diagram $x_1\hrar x^1_1\hlar x^1$ admits a pushout, which we denote by $x^1+x_1$, and in the resulting square
$$
\xymatrix{
x^1_1\ar@{^{(}->}[d]\ar@{^{(}->}[r] & x^1\ar@{^{(}->}[d] \\
x_1 \ar@{^{(}->}[r] & x^1+x_1
}
$$
all the morphisms are admissible monomorphisms. The arrow $x\epi x^2_2$ being an admissible epimorphism, it is easy to see that its 
cokernel is the arrow $ x^1+x_1\hrar x$, induced by the pushout diagram; therefore the second arrow  is an admissible monomorphism and thus
$$
 x^1+x_1\hrar x\epi x^2_2
$$
 is  an admissible short exact sequence in $\A$. It follows that $x^2_2\xrar{\sim}\dfrac{x}{x^1+x_1}$.
\vspace{0.1cm}

We can consider then the following admissible filtrations:
\begin{align*}
x^1_1\hrar x^1\hrar x^1+x_1 \hrar x \\ 
x^1_1\hrar x_1\hrar x^1+x_1 \hrar x          
\end{align*}
of $x$ in $\A$. In particular, from $x^1_1\hrar x^1\hrar x$ and $x^1_1\hrar x_1\hrar x$ we obtain that the diagram
\begin{equation}\label{qsymmetric}
\xymatrix{  
h(x^1_1)\otimes h(x^1_2)\otimes h(x^2)  \ar[d]_{\lambda\otimes\lambda}\ar[r]^{1\otimes\lambda}   &h(x^1_1)\otimes h\left(\dfrac{x}{x_1^1}\right)& h(x^1_1)\otimes h(x^2_1)\otimes h(x_2)\ar[l]_{1\otimes\lambda}\ar[d]^{\lambda\otimes 1} \\
h(x^1)\otimes h(x^2)  \ar[r]_{\lambda}   &h(x)& h(x_1)\otimes h(x_2)\ar[l]_{\lambda}
}
\end{equation}
is commutative.

\vspace{0.1cm}
On the other hand, from $x^1_1\hrar x^1+x_1 \hrar x$, and observing that $\dfrac{x^1+x_1}{x^1_1}\xrar{\sim} x^1_2\oplus x^2_1$, we get that the diagram
$$
\xymatrix{
h(x^1_1)\otimes h(x^1_2\oplus x^2_1)\otimes h(x^2_2)\ar[d]_{\lambda\otimes 1} \ar[rr]^-{1\otimes\lambda} &&h(x^1_1)\otimes h\left (\dfrac{x}{x^1_{1}}\right )\ar[d]^{\lambda}\\
h(x^1+x_1)\otimes h(x^2_2) \ar[rr]_-{\lambda}&& h(x) \\
}
$$
commutes. Similarly, taking quotients of the first filtration above by $x^1_1$ we obtain that the diagram
$$
\xymatrix{
h(x^1_2)\otimes h(x^2_1)\otimes h(x^2_2)\ar[d]_{\lambda\otimes 1} \ar[rr]^-{1\otimes\lambda} &&h(x^1_2\oplus x^2_1)\otimes h(x^2_2)\ar[d]^{\lambda}\\
h(x^1_1)\otimes h(x^2_2) \ar[rr]_-{\lambda}&& h\left (\dfrac{x}{x^1_{1}}\right ) \\
}
$$
also commutes. Since $h$ is symmetric,  the latter diagram, tensorized with $h(x^1_1)$ and compared with diagram \eqref{qsymmetric}, yields diagram \eqref{detsymmetric}. This proves the ``if" clause of the statement. For the ``only if" part, we observe that the commutative diagram \eqref{oldsymm} is just a particular case of the commutative diagram \eqref{detsymmetric}, when in  \eqref{symmetricintrinsic}, we let $x^1_1=x^2_2=0$. The proposition is proved.
\end{proof}
\begin{examples} 
(1) Let be $k$ a field and $\A=\vect$ the abelian category of finite dimensional vector spaces on $k$. Let be $G=k^*$, and let be $\Pp=\Tors(G)$. For an object $V\in\vect$, let us denote by $\Lambda^{\max}$ the top exterior power of $V$. Then we have a $G$-torsor
$$
\det(V)=\Lambda^{\max}-\{0\},
$$
called the {\it determinantal space} of $V$. For every short exact sequences of vector spaces $V'\hrar V\epi V''$, we have natural identifications 
$$
\lambda_{V',V,V''}\colon\det(V')\otimes\det(V'')\to\det(V)
$$
The collection $\{\det(V), \lambda\}_{V\in\vect}$ forms a determinantal theory on $\A$ (see \cite{ka}). This determinantal theory is non-symmetric.
\vspace{0.1cm}

(1') (Sketch) The non-symmetric determinantal theory $\det(V)$ has a symmetric analog. Let us consider the category $\Pic_k^{\Zee}$  (cf. Example \eqref{virtual} (3)). For any $V$ in $\vect$, let be $\det(V)$ the graded 1-dimensional vector space consisting of the top exterior power $\Lambda^{\max}(V)$ in degree  $\dim(V)$. Then, the correspondence
$$
V\mapsto\det(V)
$$
is a symmetric determinantal theory with values in $\Pic_k^{\Zee}$. \\
\vspace{0.1cm}

(2) {\it The universal determinantal theory.} The geometric description of the bisimplexes of $S(\A)$ of dimension $\leq 3$ has a natural interpretation in terms of the {\it universal determinantal theory}. This is a determinantal theory with values in the category of virtual objects (cf. Example \eqref{virtual} (2)):
$$
(h^u, \lambda^u)\colon\A\to V(\A)
$$
defined as follows: referring to the notations used in  Remark \eqref{bisimplexes}, for all object $a\in\A$, $h^u(a)$ is the loop $|a|$ of $S(\A)$, interpreted as on object of $V(\A)$. Given $\sigma\colon a'\hrar a\epi a''$, $|\sigma|$ is a homotopy class of homotopies between the composition of the loops $|a'|*|a''|$ and $|a|$, and it can be interpreted as an arrow
$$
\lambda^u_{\sigma}=|\sigma|\colon h^u(a')\otimes h^u(a'')\xrar{\sim}h^u(a).
$$ 
of $V(\A)$.
We claim that the pair ($h^u, \lambda^u$) defines a symmetric determinantal theory. Indeed, let be $\tau\colon a_1\hrar a_2\hrar a_3$. Interpret $|\tau|$ (see figure \eqref{Fi:tau}) as a class of homotopies between the composition of the even faces of $|\tau|$,  as in Figure \eqref{Fi:even}

\begin{figure}[h!]
\centerline{\includegraphics{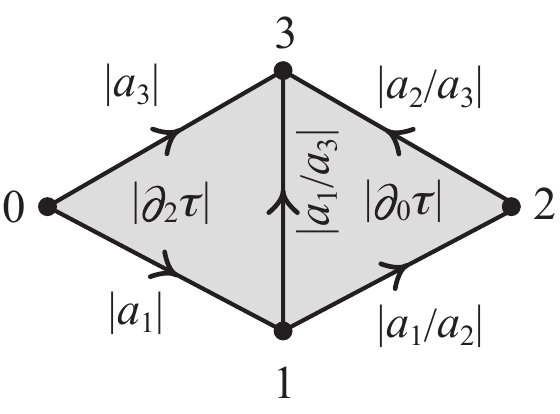}}
\caption{Even composition}\label{Fi:even}
\end{figure} 
i.e. the arrow of $V(\A)$:
$$
h^u(a_3)\xrar{(1\otimes\pt_0(\tau))\pt_2(\tau)}h^u(a_1)\otimes h^u\left (\dfrac{a_2}{a_1}\right)\otimes h^u\left(\dfrac{a_3}{a_2}\right)
$$
and the composition of the odd faces of  $|\tau|$:

\begin{figure}[h]
\centerline{\includegraphics{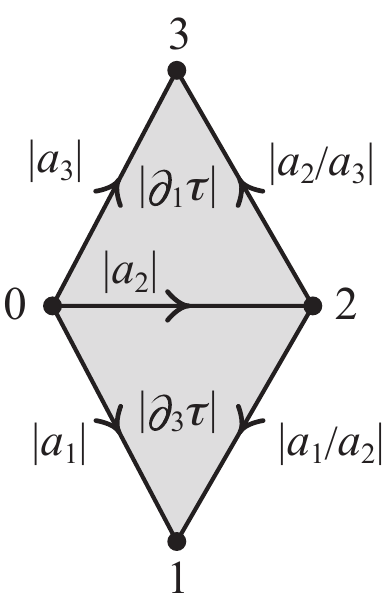}}
\caption{Odd composition}\label{Fi:odd}
\end{figure}

i.e. the arrow of $V(\A)$
$$
h^u(a_3)\xrar{(1\otimes\pt_1(\tau))\pt_3(\tau)}h^u(a_1)\otimes h^u\left (\dfrac{a_2}{a_1}\right)\otimes h^u\left(\dfrac{a_3}{a_2}\right).
$$
Thus, $|\tau|$ yields the commutativity, up to an isomorphism $\alpha$ of associativity of $V(\A)$, of a diagram of type \eqref{mult}, with $h=h^u$ and $\lambda=\lambda^u$. Similar interpretations give the functoriality of $h$ and the naturality of $\lambda^u_{\sigma}$ with respect to $\sigma$. Thus ($h^u, \lambda^u$) is a determinantal theory. A direct application of Proposition \eqref{symmdet} shows that it is symmetric. We call it the {\it universal determinantal theory} on $\A$. This terminology is justified by the following, which in the symmetric case is due 
to Deligne (cf. \cite{d}, (4.3)), and which explains how to reconstruct any $\Pp$-valued determinantal theory from the universal determinantal theory.
\begin{theorem}
(1) Let be $\Pp$ a Picard category and $\Fun^{\otimes}(V(\A), \Pp)$ the category of Picard functors $V(\A)\rar\Pp$. There exists an equivalence of categories
$$
\Det(\A,\Pp)\xrar{\sim}\Fun^{\otimes}(V(\A),\Pp).
$$
(2) Let be $\Pp$ a symmetric Picard category, and 
$\Fun_{\sigma}^{\otimes}(V(\A), \Pp)$ the category of symmetric Picard functors $V(\A)\rar\Pp$. There exists an equivalence of categories
$$
\Det_{\sigma}(\A,\Pp)\xrar{\sim}\Fun_{\sigma}^{\otimes}(V(\A),\Pp).
$$

\end{theorem}
\end{examples}

\subsection{Degree n multiplicative torsors} 
The concept of multiplicative (bi-)torsor has been introduced by Grothendieck in connection with the problem of the description of the second cohomology group of an abelian group $G$, and the classification of the central extensions of a group by an abelian group in \cite{sga7}. Although introduced in the context of group cohomology, Grothendieck's definition can be easily reworked for the more general case of simplicial sets, which is the context in which we, at first, shall use this notation.

\subsubsection{A pasting rule} Let $\mathcal C$ be the category of torsors over an abelian group $G$. This is a monoidal, strictly symmetric category. We denote by $c$ its symmetry. We introduce a notation which shall help us to write in a compact form some particular compositions of morphisms of $\mathcal C$.

\vspace{0.3cm} 

Let be $f:A\rar B_1\otimes B_2$ and $g:B_1\otimes B_3\rar C$ two morphisms of $\mathcal C$. Since $\dom(g)\neq\cod(f)$, these morphisms cannot be composed. However, we can define a new morphism, by ``pasting together" $f$ and $g$, as follows: 

\begin{equation}\label{pasting}
A\otimes B_3 \xrightarrow{f\otimes B_3}  B_1\otimes B_2\otimes B_3\xrightarrow{c\otimes B_3} B_2\otimes B_1\otimes B_3  \xrightarrow{B_2\otimes g} B_2\otimes C
\end{equation}

We shall denote such a composition by $``g\cdot f"$. Similarly, one defines $``h\cdot g\cdot f"$, and so on.

\subsubsection{Simplicial sets and pasting of  torsor morphisms} 
Let be $(\Sigma_{\bullet}, \pt_i, s_i)$ a simplicial set (as defined e.g. in \cite{gm} and \cite{gz}) and $G$ an abelian group. 

\vspace{0.2cm}

Let be given, for all $\rho\in\Sigma_{n-1}$ a $G$-torsor $T_{\rho}$ and for all $\sigma\in\Sigma_{n}$ an isomorphism of $G$-torsors:
$$
\alpha_{\sigma}:\bigotimes T_{\pt_{2i}(\sigma)}\lrar \bigotimes T_{\pt_{2i+1}(\sigma)}.
$$
Let be given $\tau\in\Sigma_{n+1}$. It is possible to construct the following composition, that we shall call the {\it even composition} of the $\alpha$'s:
$$
E_{\tau}:=``\cdots \cdot\alpha_{\pt_2(\tau)}\cdot\alpha_{\pt_0(\tau)}"
$$
and the similarly defined {\it odd composition}, as 
$$
O_{\tau}:=``\alpha_{\pt_1(\tau)}\cdot\alpha_{\pt_3(\tau)}\cdot\cdots"
$$
which are extended, respectively, to all the even/odd factors $\alpha_{\pt_i(\tau)}$ in the indicated order. 

\subsubsection{The Street decomposition of a simplex} Following Street (\cite{st}) we introduce a useful way to decompose the boundary of a simplex $\sigma$ of a simplicial set. Let be $\sigma\in\Sigma_n$. We let
\begin{align*}
\pt_+(\sigma) & =\{\pt_{2i}(\sigma)\} \\
\pt_-(\sigma) & =\{\pt_{2i+1}(\sigma)\} \\
\end{align*}
Then, we put (cf. \cite{ka2}):
\begin{align}
\pt_{++}(\sigma) &=\bigcup\pt_{+}\pt_{2i}(\sigma), & \pt_{+-}(\sigma) &=\bigcup\pt_{+}\pt_{2i+1}(\sigma) \\
\pt_{-+}(\sigma) &=\bigcup\pt_{-}\pt_{2i}(\sigma), & \pt_{--}(\sigma) &=\bigcup\pt_{-}\pt_{2i+1}(\sigma)
\end{align}

\begin{lm}
For all $\sigma\in\Sigma_n$, we have: $\pt_{++}(\sigma)=\pt_{--}(\sigma)$ and $\pt_{+-}(\sigma)=\pt_{-+}(\sigma)$.
\end{lm}

\begin{proof}
In fact, the lemma is just a restatement of the simplicial identities $\pt_i\pt_j=\pt_{j-1}\pt_i, \ (i\leq j)$. We leave the details to the reader.
\end{proof}
As a consequence, we have the
\begin{prop}
For all $n\geq 1$ the domain of the even composition coincides with the domain of the odd composition, and similarly for the target.
\end{prop}
\begin{proof}
The same argument used to prove the Lemma shows that the domain of the even composition is $\pt_{++}(\sigma)$ and the domain of the odd composition is $\pt_{--}(\sigma)$. The Lemma thus gives the identity of the domains. Similarly for the targets.
\end{proof} 

\subsubsection{Multiplicative torsors of degree $n$} 
Because of the above proposition, the following makes sense:
\begin{definition}
Let be $n\geq 1$. A {\it degree $(n-1)$-multiplicative $G$-torsor on a simplicial set  $\Sigma_{\bullet}$}  is the datum 
$T=\{T_{\rho}, \alpha_{\sigma}\}$ consisting of:
\begin{itemize}
\item{For all $\rho\in\Sigma_{n-1}$, a $G$-torsor $T_{\rho}$;}
\item{For all $\sigma\in\Sigma_n$, an isomorphism of $G$-torsors:
$$
\alpha_{\sigma}:\bigotimes T_{\pt_{2i}(\sigma)}\lrar \bigotimes T_{\pt_{2i+1}(\sigma)};
$$}
\item{For all $\tau\in\Sigma_{n+1}$, an identity
$$
E_{\tau}=O_{\tau}.
$$}
\end{itemize}
\end{definition}
\vspace{0.1cm}

Let be $\{T_{\rho}, \alpha_{\sigma}\}$ and $\{T'_{\rho}, \alpha'_{\sigma}\}$ two multiplicative $G$-torsors of degree $n-1$. A {\it morphism} between them is a collection of morphisms of the underlying $G$-torsors $f_{\rho}:T_{\rho}\lrar T'_{\rho}$, defined for all $\rho\in\Sigma_{n-1}$, such that, for all $\sigma\in\Sigma_n$, the diagram
$$
\xymatrix{
\bigotimes T_{\pt_{2i}(\sigma)} \ar[d]_{\otimes f_{\pt_{2i}(\sigma)}}\ar[r]^{\alpha_{\sigma}} & \bigotimes T_{\pt_{2i+1}(\sigma)}\ar[d]^{\otimes f_{\pt_{2i+1}(\sigma)}}  \\
\bigotimes T'_{\pt_{2i}(\sigma)}\ar[r]_{\alpha'_{\sigma}} & \bigotimes T'_{\pt_{2i+1}(\sigma)} \\
}
$$
is commutative.
\vspace{0.1cm}

The collection of the multiplicative $G$-torsors of degree $n$ over $\Sigma_{\bullet}$ and their morphisms forms a category (groupoid), that we shall denote by $\Mult_n(\Sigma_{\bullet}, G)$. It is clear that the tensor product of the underlying torsors induces a strictly symmetric tensor product also for the objects of $\Mult_n(\Sigma_{\bullet}, G)$, which in turns induces a strictly symmetric Picard category structure on $\Mult_n(\Sigma_{\bullet}, G)$, whose symmetry is defined as in $\Tors(G)$. Moreover, we have the following, whose proof we postpone to a later paper:

\begin{theorem}\label{torsorcohomol}
(a) The Picard group $\pi_0(\Mult_n(\Sigma_{\bullet}, G))$ is isomorphic to the (n+1)-th cohomology group $H^{n+1}(\Sigma_{\bullet}, G)$.
\par
(b) The group $\pi_1(\Mult_n(\Sigma_{\bullet}, G)$) is isomorphic to $Z^{n}(\Sigma_{\bullet}, G)$, the group of simplicial 
$n$-cocycles.
\end{theorem}

\begin{remarks} (1) In this paper we shall concern ourselves only with multiplicative torsors of degree 1, otherwise simply called ``multiplicative torsors". Explicitly, a multiplicative $G$-torsor on a simplicial set $\Sigma_{\bullet}$ is the data consisting of:
\begin{itemize}
\item{ For all 1-simplexes $\rho$ on $\Sigma_1$, a $G$-torsor $T_{\rho}$}
\item{For all 2-simplexes $\sigma$ on $\Sigma_2$, an isomorphism of $G$-torsors:
$$
\mu_{\sigma}:T_{\pt_0(\sigma)}\otimes T_{\pt_2(\sigma)} \stackrel{}\lrar T_{\pt_1(\sigma)},
$$ such that}
\item{For all 3-simplexes $\tau\in \Sigma_3$, a commutative diagram:
$$
\xymatrix{ 
T_{\pt_2\pt_3(\tau)}\otimes T_{\pt_0\pt_3(\tau)}\otimes T_{\pt_0\pt_1(\tau)}\ar[d]_{\mu_{\pt_3}(\tau)\otimes 1} \ar[rr]^-{1\otimes \mu_{\pt_0(\tau)}} && T_{\pt_2\pt_3(\tau)}\otimes T_{\pt_0\pt_2(\tau)}\ar[d]^{\mu_{\pt_2(\tau)}} \\
T_{\pt_1\pt_3(\tau)}\otimes T_{\pt_0\pt_1(\tau)}  \ar[rr]_-{\mu_{\pt_1(\tau)}} && T_{\pt_1\pt_2(\tau)} 
}
$$
}
\end{itemize}
(2) One defines similarly the concept of multiplicative torsor of degree n for a bisimplicial (trisimplicial, ... etc) set.
\end{remarks}

\subsection{Gerbes}
\begin{definition}
Let be $G$ an abelian group. A {\it G-gerbe} $\mathfrak{g}$,  is a category with the following extra structures and properties:
\begin{itemize}
\item{$\g$ is a connected groupoid.}
\item{For all pair of objects $x, \ y$ of $\mathfrak{g}$, the set of morphisms $Hom_{\g}(x, \ y)$ is given a structure of $G$-torsor.}
\item{For all triples of objects of $\g$, $x,  y, \ z$, the composition of morphisms in $\g$: 
$Hom_{\g}(z, \ y)\times Hom_{\g}(x, \ y)\lrar Hom_{\g}(x, \ z)$, is a $G$-bilinear map.} 
\end{itemize}
\end{definition}
As a result, the bilinear map in the definition induces an isomorphism of the following $G$-torsors:
$$
Hom_{\g}(x, \ y)\otimes Hom_{\g}(y, \ z)\simeq Hom_{\g}(x, \ z).
$$
A {\it morphism of G-gerbes} is a functor between the underlying categories which is $G$-linear on the hom-sets. It follows that a morphism of gerbes is always an equivalence.

\vspace{0.1cm}

Given two $G$-gerbes $\g$ and $\h$, we define their {\it tensor product} $\g\otimes\h$ to be the $G$-gerbe defined as: 
$$
\Ob(\g\otimes\h)=\Ob(\g)\times\Ob(\h)
$$
$$
\Hom_{\g\otimes\h}((x', x''), (y', y''))=\Hom_{\g}(x',y')\otimes\Hom_{\h}(x'', y'').
$$ 
Let be now $\Pp$ a symmetric Picard category and $\T$ a torsor over $\Pp$. Let be $x\in\Ob\T$ and denote by $\T_x$ the connected component of $\T$ containing $x$. Thus, $\T_x$ is in particular a connected groupoid.

\begin{prop}\label{pi1gerbe}
For all $x\in\Ob\T$, we have that $\T_x$ is a $\pi_1(\T)$-gerbe.
\end{prop}

\begin{proof}
Since $\T_x$ is a connected groupoid, it is enough to define an action of the group $\pi_1(\T)$ onto each hom-set 
$\Hom_{\T_x}(y,z)=\Hom_{\T}(y,z)$ and prove that this action makes each hom-set a $\pi_1(\T)$-torsor, for which the composition of morphisms is $\pi_1(\T)$-bilinear. Since in $\T_x$ the hom-sets are in bijection with each other, it suffices to prove that for all object $z$ the set  $\Aut_{\T}(z)$ is a $\pi_1(\T)$-torsor. 

\vspace{0.1cm}

Let us consider the natural isomorphisms \eqref{null} $\lambda_z\colon 1\otimes z\to z$. For all  $f\in\Aut(z)$ we have a commutative diagram:
\begin{equation}\label{action}
\xymatrix{
1\otimes z\ar[d]_{1\otimes f}\ar[r]^{\lambda_z} & z\ar[d]^f \\
1\otimes z\ar[r]_{\lambda_z} & z
}
\end{equation}
which allows us to identify $\Aut(z)$ with $\Aut(1\otimes z)$.  Define an action $\pi_1(\Pp)\times\Aut(z)$ by letting, for all 
$g\in\pi_1(\Pp)$ and $f\in\Aut(z)$, $(g,f)\mapsto g\otimes f$. Since the functor 
$$
\Pp\xrar{-\otimes z}\T
$$
is an equivalence, this action of $\pi_1(\Pp)$ onto $\Aut(z)$ is free and transitive, thus proving that $\Aut(z)$ is a $\pi_1(\Pp)$-torsor. The bifunctoriality of the action $\Pp\times\T\lrar\T$ implies the bilinearity.
\end{proof}
In particular, for each object $x\in\T$ the gerbes $\T_x$ are pairwise equivalent. Thus, Proposition \eqref{pi1gerbe} coupled with 
Proposition \eqref{inducedtorsor}, imply that the datum of a torsor over a Picard category encloses the datum of a torsor (over $\pi_0(\Pp)$) and of a gerbe (over $\pi_1(\Pp)$). 

\subsection{Multiplicative $G$-gerbes}
In analogy with the concept of ``multiplicative $G$-torsor of degree $n$" over a simplicial (bisimplicial, trisimplicial, ...) set
$\Sigma_{\bullet}$, it is also possible to introduce the concept of {\it multiplicative $G$-gerbe of degree $n$} over $\Sigma_{\bullet}$.  Since in this work we will only use multiplicative gerbes of degree 1, we shall bound ourselves to this case. 

\begin{definition}\label{defmultgerbe}
Let be $\Sigma_{\bullet}$ a simplicial set and $G$ an abelian group. A {\it multiplicative $G$-gerbe}, is the datum $(\g, \alpha, \beta)$ consisting of:

\begin{enumerate}
\item{For all $\rho\in\Sigma_{1}$, a $G$-gerbe $\g_{\rho}$;}
\item{For all $\sigma\in\Sigma_2$, an equivalence of $G$-gerbes:
$$
\alpha_{\sigma}:\g_{\pt_{2}(\sigma)}\otimes\g_{\pt_{0}(\sigma)} \xrar{\sim}\g_{\pt_{1}(\sigma)}
$$
}
\item{For all $\tau\in\Sigma_{3}$, a diagram, involving the $\alpha_{\pt\tau}$'s, commuting up to a natural 
isomorphism $\beta_{\tau}$,and written according to our pasting rule \eqref{pasting}:
$$
\beta_{\tau}:``\alpha_{\pt_2(\tau)}\cdot\alpha_{\pt_0(\tau)}"\simeq``\alpha_{\pt_1(\tau)}\cdot\alpha_{\pt_3(\tau)}"
$$
}
\item{For all $\upsilon\in\Sigma_4$, a cubic commutative diagram involving the $\beta_{\pt\upsilon}$'s, 
which can be written as 
$$
``\beta_{\pt_4\upsilon}\cdot\beta_{\pt_2\upsilon}\cdot\beta_{\pt_0\upsilon}"=
``\beta_{\pt_1\upsilon}\cdot\beta_{\pt_3\upsilon}".
$$
}
\end{enumerate}
\end{definition}

Similarly, one defines multiplicative gerbes on bisimplicial (trisimplicial, etc.) sets. 

\vspace{0.1cm}

It is possible to associate to a multiplicative gerbe of degree 1 a multiplicative torsor of degree 2:

\begin{theorem}\label{multgerbe}
A multiplicative $G$-gerbe $(\g,\alpha, \beta)$ induces a multiplicative 
$G$-torsor of degree $2$ on $\Sigma_{\bullet}$
\end{theorem}

\begin{proof}(Sketch.)
Let us consider a 2-simplex $\sigma\in\Sigma_2$. Choose elements 
$x_0\in\g_{\pt_0\sigma}$, $x_1\in\g_{\pt_1\sigma}$ and $x_2\in\g_{\pt_2\sigma}$. We define a $G$-torsor $T_{\sigma}$, associated to $\sigma$, as:
$$
T_{\sigma}:=\Hom_{\g_{\pt_1\sigma}}(\alpha_{\sigma}(x_0\otimes x_2), x_1)
$$
Condition (3) of the definition of a multiplicative gerbe implies, for all 
$\tau\in\Sigma_3$, the existence of an isomorphism of $G$-torsors:
$$
\mu_{\tau}\colon 
T_{\pt_0\tau}\otimes T_{\pt_2\tau}\rar T_{\pt_1\tau}\otimes T_{\pt_3\tau}
$$
and condition (4) shows that for all $\upsilon\in\Sigma_4$, the isomorphisms $\mu_{\pt\upsilon}$ satisfy condition (3) of the definition of 
a multiplicative torsor of degree 2. Thus, $(T, \mu)$ is a multiplicative torsor 
of degree 2.
\end{proof}
From this theorem it follows that a multiplicative gerbe of degree 1 gives rise to a class in $H^3(\simpl, G)$. More in general, a multiplicative gerbe of 
degree $n$ induces a multiplicative torsor of degree $n+1$ and thus it 
determines a class in $H^{n+2}(\simpl, G)$. 

\section{Local compactness and Grassmannians in an exact category}
Throughout this section, let $\A$ be an exact category and $\F$ its abelian envelope (cf. sect. \eqref{quillen}).
\subsection{Ind/Pro-exact categories and the Beilinson category}
We recall some facts on ind/pro objects, some of which already known, which will be useful in the sequel. We refer to 
the papers \cite{am}, \cite{B}, \cite{sga}, \cite{pre} for background on the language of ind-pro objects, exact categories and the Beilinson category $\limA$ which is the natural setting of the concepts we are going to introduce. 
\begin{definition} (\cite{B}, \cite{pre})
The category $\Pro^a(\A)$ (resp. $\Ind^a(\A)$) of the {\it strictly admissible pro-objects} (resp. {\it ind-objects}) of $\A$ is the subcategory of $\Pro(\A)$ (resp., $\Ind(\A)$) whose objects have structure morphisms which are {\it admissible} epimorphisms (resp. monomorphisms). With an abuse of language, we shall refer to $\Pro^a(\A), \Ind^a(\A)$ simply as the categories of {\it strict} ind and pro-objects of $\A$.
\end{definition}

\begin{definition} (See \cite{B}, \cite{pre}). The {\it Beilinson category} of the exact category $\A$ is the category denoted by  $\limA$ defined as the full subcategory of $\Ind^a\Pro^a(\A)$ whose objects are formal limits $\indproriga$, for $(i,j)\in\Zee\times\Zee$, such that $i\leq j$, and for which the squares 

\begin{equation}
\xymatrix{ 
X_{i'j}\ar[d]
\ar[r]
& X_{i'j'} \ar[d] \\
X_{ij}\ar[r]_{}
& X_{ij'}
}
\end{equation}
defined for $i\leq i'$, $j\leq j'$, are cartesian (and thus they are automatically cocartesian). The objects of such category will also be called {\it generalized Tate spaces}.
\end{definition}

\begin{lm} (\cite{pre}) When $(\A, \E)$ is exact, then the categories $\Ind(\A)$, $\Pro(\A)$, $\Ind^a_{\aleph_0}(\A)$, $\Pro^a_{\aleph_0}(\A)$ and $\limA$ inherit in a natural way the structure of exact categories.
\end{lm}

As an example of a category of the type $\limA$, we give the following

\begin{definition}\label{tatespaces}
Let be $k$ a field. The category $\T:=\dlim\vect$ is called the category of {\it Tate vector spaces} over $k$.
\end{definition}

Tate spaces will be our object of study in section \eqref{tate}. Objects of the Beilinson category $\limA$ provide a model for local compactness in the linear context (cf. \cite{pre}), generalizing the case of locally linearly compact vector spaces to exact categories. For this reason $\limA$ is also referred to sometimes as the {\it category of locally compact objects} over the exact category $\limA$. 

We also recall the following from \cite{B}:
\begin{prop}\label{limduality}
For any exact category $\A$, $(\limA)^o=\dlim(\A^o)$.
\end{prop}

In particular, being $\vect=\vect^o$, the category $\T$ is self-dual.

From now on, we shall bound ourselves to countable ind- and pro-categories, unless specifically stated otherwise.

\vspace{0.1cm}
 
It is also useful to recall the exact structures of the categories $\Ind(\A)$, $\Pro(\A)$ and $\limA$, by specifying the classes of their admissible mono/epimorphisms.

\begin{lm}\label{straight0} (\cite{pre}) Let be $m: X\hrar Y$ an admissible monomorphism of $\Ind(\A)$. Then for every ind-representation of the objects $X$ and $Y$, say  $X=\limiriga$, and $Y=\limjrigaY$, $m$ can be written in components as $\{m^i_j\}$ in such a way that for every $i$ there is a $j$ and an admissible monomorphism $m^i_j:X_i\hrar Y_j$. Similarly for admissible monomorphisms of $\Pro(\A)$.
\end{lm}

As a consequence of the previous Lemma, in \cite{pre}, Corollary (4.19),  we obtain:

\begin{lm}\label{straight} (Straightification of admissible monomorphisms) Let be $m: X\hrar Y$ an admissible monomorphism in $\Ind(\A)$. Then it is possible to express $X$ and $Y$ as ind-systems $X=\limjrigaX$ and $Y=\limjrigaY$, and $m=\limi m_i$, where for each $i\in I$, $m_i:X_i\hrar Y_i$ is an admissible monomorphism in $\A$. Similarly for an admissible monomorphism in $Pro^a(\A)$.
\end{lm}

The analogous propositions for admissible epimorphisms of  $\Ind^a(\A)$, $\Pro^a(\A)$ follow from the ones above.

\begin{definition} A stabilizing object $X$ of $\Ind(\A)$ ({\it resp.}, $Y\in \Pro(\A$)) is an object that can be expressed as  {$\limiriga$} for a set of objects $X_i$ ({\it resp.}, as {$\limprigaY$} for a set of objects $Y_j$), for which there exists an $i_0$ such that the morphisms $\dots\lrar X_{i-1}\lrar X_{i+1} \lrar X_{i+1}\lrar \dots $ are all isomorphisms for $i\geq i_0$. ({\it resp}., for which there exists a $j_0$ such that the morphisms $\dots\lrar Y_{j+1}\lrar Y_{j} \lrar Y_{j-1}\lrar \dots $ are isomorphisms).
\end{definition}

It is clear that a stabilizing object in $\Ind(\A)$ ({\it resp}., $\Pro(\A)$) is isomorphic to an object of $\A$.

\begin{prop}\label{stab} Let be $m: X\hrar Y$ an admissible monomorphism in $\Pro^a(\A)$. Then the quotient $Y/X$ is isomorphic to an object of $\A$ if and only if $m$ is representable by a ladder of cartesian squares.
\end{prop}

\begin{proof}
Suppose that $m$ is representable by a ladder of cartesian squares. Let's represent the two objects as $X=\limprigaXi$, and $Y=\limprigaY$. From Lemma \eqref{straight} we may assume that the monomorphism is represented by $m=\{m_j\}$, i.e. for all $j$ as an admissible monomorphism of $\A, m_j: X_j\rar Y_j$. Therefore, for every $j$ we have cartesian squares of the type:
 $$
\xymatrix{ 
X_{j+1}\ar[d]_{}
\ar[r]^{m_{j+1}}
& Y_{j+1}\ar[d]^{} \\
X_{j}\ar[r]_{m_{j}}
& Y_{j}
}$$
where the horizontal arrows are admissible monomorphisms and the vertical ones admissible epimorphisms. Then, for each $i$, it is not difficult to see that the quotients $\dfrac{Y_{i+1}}{X_{i+1}}$ and $\dfrac{Y_{i}}{X_{i}}$ are isomorphic. Thus, $Y/X=\limproi Y_i/X_i$ is stabilizing and then it lies in $\A$.

\vspace{0.2cm}

Conversely, after choosing convenient pro-systems for $U, \ V$ we can represent the monomorphism by squares like the following, where the quotients are all isomorphic to each other, e.g. via $h$ in the diagram below:
$$
\xymatrix{ 
U_i\ar[d]^{e_1}  \ar[r]^{m_2}& V_i \ar[d]^{e_2}\ar[r]^{f_2} & V_{i}/U_{i}\ar[d]^h \\
U_{i-1}\ar[r]_{m_1}& V_{i-1} \ar[r]_{f_1} & V_{i}/U_{i}\\
}
$$

We show that the square on the left is cartesian, by a diagram chase.  We can assume that the above is a diagram of abelian groups (see \cite{pre}). Let then be given $p\in V_i, \ q\in U_{i-1}$, such that $e_2(p)=m_1(q)$. We want to prove the existence of a (unique) element in $U_i$ which is a preimage of both $p$ and $q$.

\par

We have: $h(m_2(U_i)+p)=m_1(U_{i-1})+e_2(p)$. Let be $q'$ any preimage of $q$ in $U_i$. If 
$m_2(q')=p$, we are done. If not, $p-m_2(q')\in V_1$. For this element, $f_2(m_2(q')-p)=f_2m_2(q')-f_2(p)=f_2(p).$
\par
Thus, $e_2(m_2(q')-p)=m_1e_1(q')-e_2(p)=m_1(q)-e_2(p)=0.$  Put $p':=m_2(q')-p$, in $V_1$.
\par

For this element $p'$, it is then $h(m_2(U_i)+p')=m_1(U_{i-1})+e_2(p')=m_1(U_{i-1})$. 
\par

Hence, being $h$ injective, it follows $p'\in m_2(U_i)$. Then, from $m_2(q')-p=p'\in m_2(U_i)$, we finally get $p\in m_2(U_i)$: so the element $p$ has a (necessarily unique) preimage in $U_i$: therefore, the square is cartesian, and the proof is complete.

\end{proof}

\begin{cor}
Let  $X$ be an object of $\limA$, and let be $X=\indjriga$, for an ind-system of  objects $\{X_j\}$ in $\Pro^a(\A)$. Then, for $j<j'$, 
the quotient $X_{j'}/X_{j}$ is in $\A$.
\end{cor}

\subsection{Grassmannians of generalized Tate spaces}

Let be $\A$ an exact category.  We introduce some terminology. 

\begin{definition} An object of $\limA$ is called {\it compact} if it is isomorphic to an object of $\Pro^a(\A)$ and {\it discrete} if it is isomorphic to an object of $\Ind^a(\A)$.
\end{definition}

\begin{prop}\label{compdisc} If an object $Z$ is both compact and discrete, then $Z$ is isomorphic to an object of $\A$.
\end{prop}

\begin{proof}  Let be $Z\in\Ind^a\Pro^a(\A)$. Since $Z$ is discrete, it is possible to write $Z=\underset{i\in I}{``\varinjlim"}A_i$ for a system of objects $\{A_i\}$ of $\A\hrar\Pro^a(\A)$, and $Z=\underset{j\in J}{``\varprojlim"} B_j$, for a pro-system $\{B_j\}_{j\in J}$ of objects of $\A$.  In this latter expression, $Z$ is to be understood as a trivial ind-system of $\Pro^a(\A)$. By hypotheses there is an isomorphism between the two ind-pro objects represented by these two systems:
$$
\xymatrix{ 
... \ar[r] & A_{i-1} \ar[r] &A_{i} \ar[r] &...\ar[r]& A_{i+1}\ar[r] & ... \\
...  \ar[r]^{id} & Z    \ar[r]^{id} & Z  \ar[r]^{id}  &... \ar[r]^{id}   & Z \ar[r]^{id}&   ...          \\
}
$$ 

In particular, this implies that the identity $Z = Z$, which is an arrow in $\Pro^a(\A)$, factors through a composition of two 
admissible monomorphisms $Z\hrar A_k\hrar Z$ for some object $A_k\in\A$ of the pro-system $\{A_i\}$. Let be $\F$ the abelian 
envelope of $\A$. From the embedding theorem (see \cite{pre}, Theorem (6)), we have in the abelian category $\Pro(\F)$ an isomorphism 
which is the composition of two monomorphisms. Each of these monomorphism must be an isomorphism of $Z$ with an object of $
\A\hrar \F$; thus $Z$ is isomorphic to an object of $\A$, and the proposition is proved.
\end{proof}
Fix an object $X\in \limA$.   

\begin{definition}\label{G(X)} 
The {\it Sato Grassmannian} of the object $X$ is the set $\Gamma(X)$ of all the admissible subobjects $[V\hrar X]$, such that $V\in\Pro^a(\A)$ and $\dfrac{X}{V}\in\Ind^a(\A)$.
\end{definition}
In other words, for a subobject of $X$, the statement ``$[V\hrar X]\in\Gamma(X)$" means that there is an admissible short exact sequence of $\limA$: 
$$
V\hrar X\epi\dfrac{X}{V}
$$
such that $V$ is in $\Pro^a(\A)$ and $\dfrac{X}{V}$ is in $\Ind^a(\A)$.

\vspace{0.1cm}

In such a situation, and when the class of the monomorphism $m:V\hrar X$ is known, we shall simply say that $V$ is in $\Gamma(X)$.

\par

Let thus $X\in\limA$ be given through a specific ind-pro system $\{X_i\}$, $X=\limiriga$. The existence of the monomorphism 
$m:V\hrar X$ implies the existence of an $i\in I$ and an admissible monomorphism of $\Pro^a(\A)$: $m_i: V\hrar X_i$. By 
composing with the structure maps of the ind-system $\{X_i\}$, we obtain that there is an admissible monomorphism $m_j:V\hrar 
X_j$ for all $j\geq i$. Then we can write the quotient $\dfrac{X}{V}$ as 
$$
\underset{i\in I}{``\varinjlim"} \left (\dfrac{X_i}{V}\right ).
$$
The condition expressed in the definition implies that this is a strict admissible ind-system of $\A$: it results that each quotient object $\dfrac{X_i}{V}$ is then in $\A$.
 
\begin{theorem}\label{criterion} 
Let be $X\in\limA$ and $V\hrar W$ an admissible monomorphism in $\Pro^a(\A)$, with $W\in\Gamma(X)$. Then $V\in\Gamma(X)$ 
$\Leftrightarrow$ $\dfrac{W}{V}\in\A$.
\end{theorem}
\begin{proof}
($\Leftarrow$) Let be $V\hrar W\hrar X$, the composition of two admissible monomorphisms. We want to show that $\dfrac{X}{V}\in\Ind^a(\A)$.

\par

We get the following diagram, where the horizontal arrows are admissible monomorphisms, and the vertical ones admissible epimorphisms:
\begin{equation}
\xymatrix{ 
V  \ar[r]& W\ar[d]\ar[r] & X\ar[d] \\
& \dfrac{W}{V} \ar[r] & \dfrac{X}{V}\ar[d]\\
&& \dfrac{X}{W}
}
\end{equation}
In particular, we get an admissible short exact sequence $\dfrac{W}{V}\hrar \dfrac{X}{V}\epi\dfrac{X}{W}$ in $\limA$, with $\dfrac{W}{V}\in\A$ and $\dfrac{X}{W}\in\Ind^a(\A)$, since $W\in\Gamma(X)$. But $\Ind^a(\A)$ is closed under extensions in $\limA$ (cf. \cite{pre}), hence it follows $\dfrac{X}{V}\in\Ind^a(\A)$, i.e. $[V\hrar X]\in\Gamma(X)$.
\vspace{0.1cm}

($\Rightarrow$) It is clear from the same diagram. 
\end{proof}
\subsection{Partially abelian exact categories}

\begin{definition}\label{aic}
(1) Let be $(\A,\E)$ and exact category. We say that $\A$ is {\it closed under admissible intersections}, or simply that $\A$ satisfies 
the admissible intersection condition (AIC), if any pair of admissible monomorphisms with the same target, $a'\hrar a\hlar a''$ have a 
pullback $p$ in $\A$, and in the resulting diagram
$$
\xymatrix{
p\ar@{^{(}->}[d]\ar@{^{(}->}[r] & a'\ar@{^{(}->}[d] \\
a'' \ar@{^{(}->}[r] & a
}
$$
all the morphisms are admissible monomorphisms.

\vspace{0.2cm}

(2) Dually, we say that $\A$ satisfies the (AIC)$^o$, if for any pair of admissible epimorphisms with the same source: $b\epi b'$, 
$b\epi b''$, have a pushout $q$ in $\A$, and in the resulting diagram
$$
\xymatrix{
b\ar[d]\ar[r] & b'\ar[d] \\
b'' \ar[r] & q
}
$$
all the morphisms are admissible epimorphisms.
\end{definition}

\begin{lm}\label{monoext1}
Let $(\A,\E)$ be closed under admissible intersections. Consider the pullback diagram of the admissible monomorphisms $a\hrar c$, $b\hrar c$:

\begin{equation}\label{squareaic1}
\xymatrix{
p\ar@{^{(}->}[d]\ar@{^{(}->}[r]^{i} & a\ar@{^{(}->}[d] \\
b \ar@{^{(}->}[r]_{i'} & c
}
\end{equation}

Let be $j=\coker(i)$ and $j'=\coker(i')$, admissible epimorphisms. Then, there exists a unique (not necessarily admissible) monomorphism of $\A$, $m''$, making the following diagram commutative:

$$
\xymatrix{
p\ar@{^{(}->}[d]\ar@{^{(}->}[r]^{i} & a\ar@{^{(}->}[d]\ar[r]^{j} & q\ar@{^{(}-->}[d]^{m''}  \\
b \ar@{^{(}->}[r]_{i'} & c\ar[r]_{j'} & d
}
$$

\end{lm}

\begin{proof}
The above lemma holds in any abelian category. It is thus valid in the abelian envelope $\F$ of $\A$. In particular, we get that $m''$ is a monomorphism of $\A$.
\end{proof}
\begin{lm}\label{monoext2}
In the situation of Lemma \eqref{monoext1}, let us extend  diagram \eqref{squareaic1} to a $3\times 3$ diagram by passing to the cokernels in the abelian envelope $\F$:
\begin{equation}\label{squareaic2}
\xymatrix{
p\ar@{^{(}->}[d]_{m'}\ar@{^{(}->}[r]^{i} & a\ar@{^{(}->}[d]^{m}\ar[r]^{j} & q\ar@{^{(}->}[d]^{m''}  \\
b\ar@{^{(}->}[r]_{i'}\ar[d]_{e'} & c\ar[d]^{e}\ar[r]_{j'} & d\ar[d]^{e''} \\
r\ar[r]_{i''} & s\ar[r]_{j''} & t
}
\end{equation}
where $(m'',e'')$ and $(i'', j'')$ are short exact sequences in  $\F$, while $(i,j), \ (i',j'), \ (m,e), \ (m',e')$ are admissible short exact sequences in $\A$. In this case, the bottom right square is a pushout diagram.
\end{lm}

\begin{proof}
The proof is a direct verification of the universal property of pushouts relative to the bottom right square.
\end{proof}

 \begin{prop}\label{selfdual}
 If $\A$ satisfies both (AIC) and (AIC)$^o$, then in diagram \eqref{squareaic2} $m''$ is an admissible monomorphism and $e''$ an admissible epimorphism. As a result, \eqref{squareaic2} represents an object of the category $S_2S_2(\A)$, of the delooping 
$S_{\bullet}S_{\bullet}(A)$ of $S_{\bullet}(\A)$ (cf. sect. \eqref{delooping}).
\end{prop}

\begin{proof}
From Lemma \eqref{monoext2} we know that the diagram
$$
\xymatrix{
c\ar[d]_e\ar[r]^{j'} & d\ar[d]^{e''} \\
s \ar[r]_{j''} & t
}
$$
is a pushout diagram of the admissible epimorphisms $e$ and $j'$. Since $\A$ satisfies the (AIC)$^o$, it follows that $t\in\A$ and that $j'',\ e''$ are admissible epimorphisms. Then, from Lemma \eqref{admono}, $m''=\ker(e'')$ is an admissible monomorphism. 
\end{proof} 
\begin{definition}\label{qabelian}
An exact category $(\A,\E)$ is called {\it partially abelian exact} if every arrow $f$ which is the composition of an admissible 
monomorphism followed by an admissible epimorphism can be factored in a unique way as the composition of an admissible 
epimorphism followed by an admissible monomorphism.
\end{definition}
 For example, an abelian category is partially abelian exact.
 \begin{theorem}\label{factor}
 The category $(\A,\E)$ is partially abelian exact if and only if $\A$ satisfies both (AIC) and (AIC)$^o$.
 \end{theorem}
 \begin{proof}
 We first show that if $\A$ satisfies (AIC) and (AIC)$^o$, then $\A$ is partially abelian exact. Let be $f$ the composition $x\stackrel{m}\hrar y\stackrel{e}\epi z$  of the admissible mono $m$ followed by the admissible epi $e$. Let be $k\hrar y$ the kernel of $e$, which is an admissible monomorphism, and consider the pullback $p$ of $k\hrar y$ and $m$. We get the following diagram:
$$
\xymatrix{
p\ar@{^{(}->}[d]_{m'}\ar@{^{(}->}[r]^{} & k\ar@{^{(}->}[d]  \\
x \ar[d]_{e'}\ar@{^{(}->}[r]_{m} & y\ar[d]_{e} \\
z'\ar@{-->}[r]_{h} & z
}
$$ 
 in which $m':p\hrar x$ is an admissible monomorphism and $e'=\coker(m')\in\A$. From Lemma \eqref{monoext1} there exists a 
 unique admissible monomorphism $h:z'\hrar z$ for which the bottom square commutes. Thus, $x\stackrel{e'}\epi z'\stackrel{h}
 \hrar z$ is the required factorization of $f$.
 
 \vspace{0.2cm}
 
Conversely, supppose $A$ partially abelian exact. We first show that $\A$ satisfies (AIC). Let be a diagram of admissible 
monomorphisms of $\A$: $k\hrar y\hlar x$. Let be $z\colon=\dfrac{y}{k}$ and apply the factorization condition to the composition 
$x\hrar y\epi z$. We get the diagram 
$$
\xymatrix{
                                     & k\ar@{^{(}->}[d]^m  \\
x \ar[d]_{e'}\ar@{^{(}->}[r]_{} & y\ar[d]_{e} \\
y\ar@{^{(}->}[r]_{i''} & z
}
$$ 
where $e'$ is an admissible epi and $i''$ an admissible monomorphism.
 
 \par
 
 From the universal property of $m=\ker(e)$, we obtain a unique morphism $k'\rar k$, for which the following diagram 
\begin{equation}\label{squareaic3}
\xymatrix{
k'\ar[r]^i\ar@{^{(}->}[d]_{m'} & k\ar@{^{(}->}[d]^m  \\
x \ar[d]_{e'}\ar@{^{(}->}[r]_{} & y\ar[d]^{e} \\
y\ar@{^{(}->}[r]_{i''} & z
}
\end{equation}
 is commutative.
 
 \vspace{0.1cm}
It is clear that $i$ is a monomorphism in $\F$, hence in $\A$, and that the top square is cartesian. We want to prove that $i$ is an admissible monomorphism. Consider the admissible epimorphism $j'=\coker(i')$, and the epimorphism of $\F$ $j=\coker(i)$. Since the top square of \eqref{squareaic3} is cartesian in $\F$, we obtain, from Lemma \eqref{monoext1}, a unique monomorphism of $\F$, $m''\colon\dfrac{k}{k'}\hrar\dfrac{y}{x}$ making the following diagram
$$
\xymatrix{
k'\ar@{^{(}->}[d]\ar@{^{(}->}[r]^{i} & k\ar@{^{(}->}[d]^{m}\ar[r]^{j} & k/k'\ar@{^{(}->}[d]^{m''}  \\
x \ar@{^{(}->}[r]_{i'} & y\ar[r]_{j'} & y/x
}
$$
commutative. 

\vspace{0.1cm}
 
Let us call, in the previous diagram, $f$ the composition $j'\cdot m$.  Since $m$ is an admissible monomorphism, and $j'$ an 
admissible epimorphism, we can factor $f$ as a composition  
$k\stackrel{a}\epi z\stackrel{b}\hrar\dfrac{y}{x}$ where $a$ is an admissible epimorphism and $b$ an admissible monomorphism. In 
the abelian envelope $\F$ we thus obtain two decompositions of $f$ as an admissible epimorphism followed by an admissible 
monomorphism. Since in an abelian category every arrow has an essentially unique such decomposition, it must be 
$\dfrac{k}{k'}\stackrel{\sim}\rar z$, and $j$ is an admissible epimorphism. Since $i=\ker(j)$, it follows from Lemma \eqref{admono} 
that $i$ is an admissible monomorphism, as required.

\vspace{0.1cm}

By duality, we prove that $\A$ satisfies also (AIC)$^o$. This conlcudes the proof of the Theorem.
 \end{proof}
 \subsection{Grassmannians and intersections}
 In this and in the next section we clarify the behavior of $\Gamma(X)$ under admissible short exact sequences of $\limA$. The main result is theorem \eqref{projlift}, which roughly speaking allows us to lift an element $U\in\Gamma(X)$ along admissible monomorphisms $Y\hrar X$ and to project it along admissible epimorphisms $X\epi Z$ of $\limA$ to elements of the Grassmannians of $Y$ and $Z$, respectively, under the assumption that $\A$ is partially abelian exact. We start by showing that $\Gamma(X)$ is closed under the operation of taking the intersections of two of its elements, under the condition that $\A$ satisfies the (AIC).

\begin{theorem}\label{limintersections} 
Let be $\A$ an exact category satisfying (AIC). Let be $X\in\limA$ and $[U\hrar X], [V\hrar X]\in\Gamma(X)$. For all $m\colon U\hrar X, n\colon V\hrar X$ in their respective equivalence classes, the diagram
\begin{equation}\label{prolimpullback}
\xymatrix{
 U \ar@{^{(}->}[r]^m & X & V\ar@{_{(}->}[l]_n  \\
}
\end{equation}
can be completed to a pullback diagram in $\limA$:
\begin{equation}\label{intgr}
 \xymatrix{
U\cap V\ar@{^{(}->}[d]\ar@{^{(}->}[r]   & U\ar@{^{(}->}[d] \\
V \ar@{^{(}->}[r] & X
}
\end{equation}
such that $U\cap V\hrar U$ and $U\cap V\hrar V$ are admissible monomorphisms, and, after composing the arrows of \eqref{intgr}, we get: $[U\cap V\hrar X]$ is in $\Gamma(X)$.
\end{theorem}
\begin{lm}\label{pullbackmorphism}
Let be $\A$ an exact category satisfying the (AIC), and let $\F$ be its abelian envelope.  Suppose we have two pullback diagrams 
 \begin{equation}\label{pullbacks}
\xymatrix{
P \ \ar@{^{(}->}[d]_{j_{1}}\ar@{^{(}->}[r]^{i_{1}} & A\ar@{^{(}->}[d]^{j_2}  && P''\ \ar@{^{(}->}[d]_{l_{1}}\ar@{^{(}->}[r]^{k_{1}} & A''\ar@{^{(}->}[d]^{l_2}  \\
B\ar@{^{(}->}[r]_{i_{2}} & Z && B''\ar@{^{(}->}[r]_{k_{2}} & Z'' \\
}
\end{equation}
where all the morphisms are admissible monomorphisms, and there are admissible epimorphisms 
 $e\colon A\epi A'', f\colon Z\epi Z'', g\colon B\epi B''$ for which we have a commutative diagram 
 \begin{equation}\label{cubic1}
\xymatrix{
& P\ar@{^{(}->}[rr]  \ar@{^{(}->}'[d]|{\;\;\;\;\; }[dd]                              &     &    A \ar@{^{(}->}[dd]^{}\ar@{->>}[dl]^{e}\\ 
                        P'' \ar@{^{(}->}[rr]\ar@{^{(}->}[dd]            &      &    A''  \ar@{^{(}->}[dd]|{} \\ 
                        &           B \ar@{^{(}->}'[r][rr]\ar@{->>}[dl]^{g}          &      &     Z  \ar@{->>}[dl]^{f} \\  
                        B'' \ar@{^{(}->}[rr]                                        &      &     Z''
                }
\end{equation}        
such that the square 
\begin{equation}\label{floor}
\xymatrix{
B \ \ar@{->>}[d]_{g}\ar@{^{(}->}[r]^{} & Z\ar@{->>}[d]^{f}   \\
B''\ar@{^{(}->}[r]_{} & Z''  \\
}
\end{equation}
is admissible and cartesian. Then there exists a unique morphism $r\colon P\rar P''$ for which the above cubic  diagram commutes, and $r$ is an admissible epimorphism.
 \end{lm}
 \begin{proof}
 The existence and uniqueness of $r$ is a consequence of the universal property of the pullback $P''$. We now prove that $r$ is an admissible epimorphism, by showing that $r$ is an epimorphism of $\F$ whose kernel is in $\A$, and then applying Lemma \eqref{admono}.
\vspace{0.1cm}
 
Let thus consider diagram \eqref{cubic1} in the abelian envelope $\F$. To prove that $r$ is an epimorphism of $\A$, we use a diagram-chase argument. 

\vspace{0.1cm}

Suppose, therefore, that an element $a$ is given in $P''$. We want to construct a preimage of $a$ through $r$.

\vspace{0.1cm}

Construct, from $a$, the following elements: $d=k_1(a)\in A''$; $c=l_1(a)\in B''$ and $e=k_2(c)\in Z''$. Then, lift $e$ to a preimage $\tilde e$ in $Z$, which exists since $f$ is surjective. From the cartesianity of diagram \eqref{floor}, we get a unique element $b\in B$ such that $i_2(b)=e$ and $g(b)=c$.

\vspace{0.1cm}

Next, consider the preimages $\tilde d$ of $d$ in $A$. If there exists a $\tilde d$ such that $j_1(\tilde d)=\tilde e$, then, from cartesianity of the left square in \eqref{pullbacks}, we obtain a unique element $\tilde x\in P$ for which $i_1(\tilde x)=\tilde d$ and $j_1(\tilde x)=b$. Thus, $r(\tilde x)$ is the unique preimage of $d$ in $P$, and $r$ is surjective in this case.

\vspace{0.1cm}

Suppose, on the other hand, that for all elements $\tilde d$ preimages of $d$ in $A$, it is $j_2(\tilde d)\neq\tilde e$. In this case, pick any $j_2(\tilde d)$ in $Z$. We get:
$$
f(j_2(\tilde d)-\tilde e)=(f\cdot j_2)(\tilde d)-f(\tilde e)=e-e=0.
$$
For a given $\tilde e'=j_2(\tilde d)$ in $Z$, and from cartesianity of diagram \eqref{floor}, we obtain a unique element $\tilde b$ in $B$, such that $i_1(\tilde b)=\tilde e'$ and $g(\tilde b)=c$. Now, the cartesianity of the left square in \eqref{pullback} yields again a unique element $\tilde x$ in $P$ for which $j_1(\tilde x)=\tilde b$ and $i_1(\tilde x)=\tilde d$. We thus have, again:
$$
k_1\cdot r(\tilde x)=(e\cdot i_1)(\tilde x)=d
$$
and we reduce ourselves to the previous case. Then, $r(\tilde x)=a$, and $r$ is an epimorphism.

\vspace{0.1cm}

We now prove that $r$ is an admissible epimorphism. For this, we consider the following double cubic diagram, which is the extension of the cubic diagram \eqref{cubic1} to the kernels of the epimorphisms there involved. 

\vspace{0.1cm}

Let then be $\gamma=\ker(g)$; $\epsilon=\ker(e)$; $\phi=\ker(f)$ admissible monomorphisms, and $s=\ker(r)$, the kernel of $r$ in $\F$. Then, we get the following cubic diagram in $\F$:
$$
\xymatrix{
&    A'\ar@{^{(}->}[rr]^{} \ar@{^{(}->}'[d][dd]^{\epsilon}  &     &Z'\ar@{^{(}->}[dd]^{\phi}\\ 
P'\ar@{^{(}->}[rr]^{\;\;\;\;\;\;\;\;\;\; \;\;\;\;\;  n}\ar@{^{(}->}[dd]_{s} \ar@{^{(}->}[ur]^{}    &   &   B'\ar@{^{(}->}[ur]^{}\ar@{^{(}->}[dd]^{\gamma} \\ 
 &           A\ar@{^{(}->}'[r][rr]^{}\ar@{->>}'[d][dd]^{e}   &      &            Z\ar@{->>}[dd]^{f} \\  
P\ar@{->>}[dd]_{r}\ar@{^{(}->}[rr]^{\ \ \ \ \ \ \ \ \  j_1}\ar@{^{(}->}[ur]^{}                    &     &                   B\ar@{->>}[dd]^g\ar@{^{(}->}[ur]_{} \\
&           A''\ar@{^{(}->}'[r][rr]   &      &            Z'' \\ 
 P''\ar@{^{(}->}[ur]^{}\ar@{^{(}->}[rr]_{l_1}  &   &    B''\ar@{^{(}->}[ur]^{} \\
                }
$$
In this diagram the arrows composing the top square are monomorphisms induced by the universal properties of the kernels involved. The columns $(\epsilon, e)$, $(\phi, f)$, $(\gamma, g)$ are admissible short exact sequences of $\A$, while the column $(s,r)$ is a short exact sequence of $\F$.  In order to prove that $r$ is an admissible epimorphism, we shall prove that $s=\ker(r)$ is an admissible monomorphism. The claim will then follow from Lemma \eqref{admono}. It will be enough to show that the square 
\begin{equation}\label{sismono}
\xymatrix{
P' \ \ar@{^{(}->}[d]_{s}\ar@{^{(}->}[r]^{} & B'\ar@{^{(}->}[d]^{\gamma}   \\
P_{}\ar@{^{(}->}[r]_{j_1} & B  \\
}
\end{equation}
is cartesian in $\F$. In fact, this will imply that it is the pullback square of two admissible monomorphisms of $\A$, i.e. $\gamma$ and $j_1$, thus, from the (AIC) condition and Lemma \eqref{pullback}, the square is cartesian in $\A$ and $s$ is an admissible monomorphism.

\vspace{0.1cm}

We shall use also in this case a diagram-chase argument. Let be $v$ in $P$ and $u$ in $B'$, two elements such that $j_1(v)=\gamma(u)=w$ in $B$. This element $w$ is sent by $g$ into 0 of $B''$, since $g(w)=g\cdot\gamma(u)$ and $(\gamma, g)$ is an admissible short exact sequence. Thus, $l_1\cdot r(v)=0$. But $l_1$ is a monomorphism, so $r(v)=0$.

\vspace{0.1cm}

It follows that $v$ belongs to the kernel of $r$, hence to the image of $s$. Let thus be $x$ in $P'$ the unique element such that $s(x)=v$. The cartesianity of \eqref{sismono} is proved if we can show that $n(x)=u$. But this is clear, since all the morphisms involved in diagram \eqref{sismono} are monomorphisms, and since $\gamma(u)=w=j_1\cdot s(x)$, then $n$ must send $x$ into the unique preimage of $w$ in $B'$, i.e. $u$. Thus, \eqref{sismono} is cartesian in $\F$. Then $r$ is an admissible epimorphism and the proof of the Lemma is complete. 
\end{proof}
We notice that the proof of the claim that $r$ is an epimorphism also proves the following
\begin{cor}\label{resulting}
The resulting admissible square in diagram \eqref{cubic1}:
$$
\xymatrix{
P'' \ \ar@{->>}[d]_{r}\ar@{^{(}->}[r]^{} & A''\ar@{->>}[d]^{}   \\
P_{}\ar@{^{(}->}[r]_{} & A  \\
}
$$
is cartesian.
\end{cor}
\begin{prop}\label{intersections1}
Let be $U, V, Z$ objects of $\Pro^a(\A)$, with admissible monomorphisms $U\hrar Z$ and $V\hrar Z$, which can be expressed as ladders of cartesian admissible squares of $\A$.   Then the pullback 
$U\times_Z V$ exists in $\Pro^a(\A)$ and in the resulting diagram
 $$
 \xymatrix{
U\times_Z V\ar[d]\ar[r]& U
 \ar[d]^m \\
V \ar[r]_n & Z
}
$$
 the morphisms $U\times_Z V\rar U$ and $U\times_Z V\rar V$ are admissible monomorphisms which can be expressed as ladders of cartesian admissible squares of $\A$.
\end{prop}
\begin{proof}
By Lemma \eqref{straight}, we can write, in components of $\A$, $Z=\proZj$, $U=\proUj$ and $V=\proVj$, so that $m$ and $n$ are given by cartesian ladders of admissible monomorphisms $m_j:U_j\hrar Z_j$, $n_j:V_j\hrar Z_j$. Then, the object
$$
\limproj U_j\times_{Z_j} V_j\in\Pro^a(\A)
$$
is clearly the pullback $U\times_Z V$. By construction, the morphisms $U\times_Z V\hrar U$ and $U\times_Z V\hrar V$ are admissible monomorphisms. They are represented by cartesian ladders by  Corollary \eqref{resulting}. The proof is complete.
 \end{proof}
 
 Let us now consider the case for $\limA$. Let be $X\in\limA$ and suppose $X=``\limj"X_j$, for $X_j\in\Pro^a(\A)$.
 
 \vspace{0.1cm}
 
 Let be $U\in\Pro^a(\A)$, and $m:U\hrar X$ an admissible monomorphism in $\limA$. Since $U$ in $\limA$ is represented as a trivial ind-system of $\Pro^a(\A)$, the datum of $m$ is equivalent to the datum of the existence of an index $j$ and an admissible monomorphism $U\hrar X_j$ in $\Pro^a(\A)$.
 \vspace{0.1cm}
 
 Then, if $[U\hrar X], [V\hrar X]\in\Gamma(X)$, there are indexes $j_1, j_2$ for which any pair of representatives $m\colon U\hrar X$, $n\colon V\hrar X$ are given in components by admissible monomorphisms of $\Pro^a(\A)$, $U\hrar X_{j_1}, V\hrar X_{j_2}$, as ladders of cartesian squares in $\A$. By taking $j=\max(j_1, j_2)$, we can assume, without loss of generality, $j_1=j_2=j$.
 \begin{lm}\label{intersections2}
 With the same notation as above, the object $U\times_{X_j} V$ is a pullback in $\limA$ of the diagram\eqref{prolimpullback}.
\end{lm}
\begin{proof}
The object $U\times_{X_j} V$ exists from Proposition \eqref{intersections1}, and it is a pullback of the diagram $U\hrar X_j\hlar V$.
Since $X=``\limj"X_j$ is a strictly admissible ind-system, and $U\hrar X_j$, $V\hrar X_j$ are admissible monomorphisms, by composition we get admissible monomorphisms $U\hrar X_{j'}$, $V\hrar X_{j'}$ for  all $j'\leq j$, still represented as ladders of cartesian admissible squares. It is easy to check that 
$U\times_{X_{j'}} V\stackrel{\sim}\rar U\times_{X_j} V$. Thus,
$$
``\limj" U\times_{X_j} V\stackrel{\sim}\rar U\times_{X_j} V
$$
so that $U\times_{X_j} V$ is a pullback of \eqref{prolimpullback}
\end{proof} 
 
We shall denote the object $``\limj" U\times_{X_j} V$ by $U\cap V$.
 
\vspace{0.1cm}

We now can prove Theorem \eqref{limintersections}.

{\it Proof of Theorem \eqref{limintersections}.}
In Lemma \eqref{intersections2} we have proved the existence, under the assumptions of the Theorem, of a pullback square \eqref{intgr}, where $U\cap V\in\Pro^a(\A)$, and $U\cap V\hrar U, U\cap V\hrar V$ admissible monomorphisms from Proposition \eqref{intersections1}. It remains to prove that $[U\cap V\hrar X]\in\Gamma(X)$.

\vspace{0.1cm}

This can be achieved by the consideration of the induced admissible short exact sequence of $\limA$:
$$
 \xymatrix{
\dfrac{U}{U\cap V}\ar@{^{(}->}[r] & \dfrac{X}{U\cap V}\ar@{->>}[r] & \dfrac{X}{U} \\
}
$$
Since from Proposition \eqref{intersections1} the monomorphism $U\cap V\hrar U$ can be expressed as a ladder of cartesian squares, the quotient $\dfrac{U}{U\cap V}$ is in $\A$.

\par

On the other hand, since $[U\hrar X]$ is in $\Gamma(X)$, $\dfrac{X}{U}$ is in $\Ind^a(\A)$. Thus, in the above short exact sequence, 
the first and the last term are in $\Ind^a(\A)$, which, being closed under extensions in $\limA$, forces $\dfrac{X}{U\cap V}$ to be 
also in $\Ind^a(\A)$. Then $[U\cap V\hrar X]$ is in $\Gamma(X)$, and the Theorem is proved.
\qed
\subsection{Grassmannians and short exact sequences }

We now discuss the behavior of Sato Grassmannians under admissible short exact sequences of $\limA$.

\begin{prop}\label{lift}
Let be $\A$ an exact category satisfying the (AIC). Let be $m\colon X\hrar Y$ an admissible monomorphism in $\limA$, and $[U\hrar Y]$ 
an element of $\Gamma(Y)$. Then the diagram
\begin{equation}\label{xyu}
\xymatrix{
   X \ar@{^{(}->}[r]^m & Y \\
                                    & U\ar@{^{(}->}[u]
}
\end{equation}
can be completed to a pullback diagram 
\begin{equation}\label{leftsquare}
\xymatrix{
   X \ar@{^{(}->}[r]^m & Y \\
U\cap X\ar@{^{(}->}[r]\ar@{^{(}->}[u]  & U\ar@{^{(}->}[u]
}
\end{equation}
where the object $U\cap X$ is in $\Pro^a(\A)$, all the maps are admissible monomorphisms, and the resulting composition $[U\cap X\hrar X]$ is in $\Gamma(X)$.
\end{prop}
\begin{proof}
Straightify $m$. Then, for all $j$, we can represent  $m$ by a system of monomorphisms of $\Pro^a(\A)$:
$$
\xymatrix{ 
... \ar@{^{(}->}[r] & X_{i-1} \ar@{^{(}->}[d]\ar@{^{(}->}[r] &X_{i}\ar@{^{(}->}[d] \ar@{^{(}->}[r] &X_{i-1} \ar@{^{(}->}[d]\ar@{^{(}->}[r]&... \\
... \ar@{^{(}->}[r] & Y_{i-1} \ar@{^{(}->}[r] &Y_i \ar@{^{(}->}[r] &Y_{i+1} \ar@{^{(}->}[r]&... \\
}
$$
Since $U\in\Pro^a(\A)$, the existence of the admissible monomorphism $U\hrar Y$ in $\limA$ is equivalent to the existence of an $i$ and of an admissible monomorphism $U\hrar Y_i$ of $\Pro^a(\A)$. For this monomorphism we have the diagram, in $\Pro^a(\A)$:
\begin{equation}\label{xi}
\xymatrix{
 X_i \ar@{^{(}->}[r] & Y_i & U\ar@{_{(}->}[l]  \\
}
\end{equation}
We straightify this diagram writing, in components: $X_i=\proXj$, $Y_i=\proYj$, $U=\proUj$. We then obtain a diagram of objects of $\A$, as follows:
$$
\xymatrix{ 
... \ar@{->>}[r] & X_{j+1} \ar@{^{(}->}[d]\ar@{->>}[r] &X_{j}\ar@{^{(}->}[d] \ar@{->>}[r] &X_{j-1} \ar@{^{(}->}[d]\ar@{->>}[r]&... \\
... \ar@{->>}[r] & Y_{j+1} \ar@{->>}[r] &Y_j \ar@{->>}[r] &Y_{j-1} \ar@{->>}[r]&... \\
... \ar@{->>}[r] & U_{j+1} \ar@{^{(}->}[u]\ar@{->>}[r] &U_{j}\ar@{^{(}->}[u] \ar@{->>}[r] &U_{j-1} \ar@{^{(}->}[u]\ar@{->>}[r]&... \\
}
$$
In this diagram, the horizontal arrows are admissible epimorphisms, the vertical arrows admissible monomorphisms, and the square corresponding to the morphism $U\hrar Y$ are cartesian.

\vspace{0.1cm}

We then construct, for each $j$, the pullback of 
$$ 
\xymatrix{
 X_j \ar@{^{(}->}[r]^m & Y_j & U_j\ar@{_{(}->}[l]_n  \\
}
$$
which exists since $\A$ satisfies the (AIC). We are now in the hypotheses of Lemma \eqref{pullbackmorphism}; its application gives us a strictly admissible pro-system $\{X_j\times_{Y_j}U_j\}_j$, and then we get an object $\limproj X_j\times_{Y_j} U_j$, which is a pullback in $\Pro^a(\A)$ of the diagram \eqref{xi}. 

\vspace{0.1cm}

Let us denote this pullback by $X_i\times_{Y_i} U$. For all $i\leq j$ we have a canonical map of corresponding pullbacks, induced by the diagram:
 \begin{equation}\label{cubic2}
\xymatrix{
& X_i\times_{Y_i}U\ar@{^{(}->}[rr]  \ar@{^{(}->}'[d][dd]\ar@{^{(}->}[ld]  &     &    X_i \ar@{^{(}->}[dd]^{}\ar@{^{(}->}[dl]\\ 
                        X_j\times_{Y_j}U \ar@{^{(}->}[rr]\ar@{^{(}->}[dd]   &      &    X_j \ar@{^{(}->}[dd]|{} \\ 
                        &           U \ar@{^{(}->}'[r][rr]\ar[dl]^{id}       &      &     Y_i  \ar@{^{(}->}[dl]^{} \\  
                        U \ar@{^{(}->}[rr]                                        &      &     Y_j
                }
\end{equation}   
and it is clear that such arrow is a monomorphism. Then, the object $\limi (X_i\times_{Y_i} U)$ is the pullback of the diagram \eqref{xyu}. We shall denote this object by $U\cap X$. 

\vspace{0.1cm}

A priori, $U\cap X$ is an object of $\Ind\Pro^a(\A)$. However, for each $i$, we have from the above cubic diagram an admissible monomorphism in $\Pro^a(\A)$: 
$$
X_i\times_{Y_i} U\stackrel{m_i}\hrar U.
$$
Therefore, the admissible monomorphisms $\{m_i\}$ form an inductive system of admissible monomorphisms, which gives raise to an admissible monomorphism
$$
\limi (X_i\times_{Y_i} U)\hrar U.
$$
But $U$ is in $\Pro^a(\A)$, and then the object $U\cap X=\limi (X_i\times_{Y_i} U)$ also belongs to $\Pro^a(\A)$. We therefore get a cartesian diagram of type \eqref{leftsquare}, where all the morphisms are admissible monomorphisms, and $U\cap X\in\Pro^a(\A)$.

\vspace{0.1cm}

It is left to prove that $[U\cap X\hrar X]$ is in $\Gamma(X)$.  We argue as in the proof that $U\cap X$ is in $\Pro^a(\A)$. The above square being cartesian, we get, on the quotients, a monomorphism:
$$
\dfrac{X}{U\cap X}\hrar \dfrac{Y}{U}
$$
A priori, the object $\dfrac{X}{U\cap X}$ is in $\Ind^a\Pro^a(\A)$. But since $[U\hrar Y]\in\Gamma(Y)$, 
$\dfrac{Y}{U}$ is in $\Ind^a(\A)$. Thus, $\dfrac{X}{U\cap X}$ is in $\Ind^a(\A)$, i.e. $[U\cap X\hrar X]\in\Gamma(X)$. Proposition
\eqref{lift} is proved.

\end{proof}
Thus, for an admissible monomorphims $X\hrar Y$ in $\limA$, given $U\in\Gamma(Y)$, in order to prove that the ``intersection" $U\cap X$ is an element of the Grassmannians of $X$, is sufficient to assume that $\A$ satisfies the (AIC). However, to make sure that the quotient $\dfrac{U}{U\cap X}$  is an element of the Grassmannians of the quotient object $\dfrac{Y}{X}$, we need also the dual condition (AIC)$^o$. This is the content of the next statement.
\begin{theorem}\label{projlift}
Let be $\A$ a partially abelian exact category. Let be 
$$
 \xymatrix{
X\ar@{^{(}->}[r] & Y\ar@{->>}[r] & Z \\
}
$$
an admissible short exact sequence of $\limA$, and let $[U\hrar Y]$ in $\Gamma(Y)$ be given. Then we have a commutative diagram:
\begin{equation}\label{lp}
\xymatrix{
U\cap X \ar@{^{(}->}[r]\ar@{^{(}->}[d] &U \ar@{->>}[r]\ar@{^{(}->}[d] &\dfrac{U}{U\cap X}\ar@{^{(}->}[d]^m\\
 X \ar@{^{(}->}[r] &Y\ar@{->>}[r] &Z  \\
}
\end{equation}
in which the top sequence is an admissible short exact sequence of $\Pro^a(\A)$, such that the arrow $\dfrac{U}{U\cap X}\hrar Z$ is an admissible monomorphism and $[\dfrac{U}{U\cap X}\hrar Z]$ is in $\Gamma(Z)$.

\end{theorem}

{\bf Terminology.} We shall say that $U$ has been {\it lifted to $X$}  along the admissible monomorphism $X\hrar Y$, and that  $U$ has been {\it projected to $Z$} along the corresponding epimorphism $Y\epi Z$.
\begin{proof}
Let us keep the same notations as in the proof of Prop. \eqref{lift}. As we have seen, the diagram \eqref{leftsquare} is constructed from the diagrams \eqref{xi} of $\Pro^a(\A)$, by forming the limit $\limi (X_i\times_{Y_i} U)$, which is still an object of $\Pro^a(\A)$. Let us take the quotients of the horizontal monomorphisms and get the following diagram, where the horizontal sequences are admissible short exact:
\begin{equation}\label{m_i}
\xymatrix{ 
U\times_{Y_i} X_i \ar@{^{(}->}[r]\ar@{^{(}->}[d] &U \ar@{->>}[r]\ar@{^{(}->}[d] &\dfrac{U}{U\times_{Y_i} X_i} \\
 X_i \ar@{^{(}->}[r] &Y_i\ar@{->>}[r] &Z_i  \\
}
\end{equation}
We now prove the existence of an admissible monomorphism $m_i:\dfrac{U}{U\times_{Y_i} X_i}\hrar Z_i$ making \eqref{m_i} commutative.

\vspace{0.1cm}

As in Prop. \eqref{lift}, write $X_i=\proXij$, $Y_i=\proYij$, $U=\proUj$, with $X_{i,j}, Y_{i.j}, U_j$ objects of $\A$.

\vspace{0.1cm}

For all $j$ we then have cartesian diagrams
$$
\xymatrix{ 
U_j\times_{Y_{i,j}} X_{i,j} \ar@{^{(}->}[r]\ar@{^{(}->}[d] &U_j \ar@{^{(}->}[d] \\
 X_{i,j} \ar@{^{(}->}[r] &Y_{i,j}  \\
}
$$
Since $\A$ is partially abelian, we can apply Proposition \eqref{selfdual}. and we get commutative diagrams for all $j$
$$
\xymatrix{ 
U_j\times_{Y_{i,j}} X_{i,j} \ar@{^{(}->}[r]\ar@{^{(}->}[d] &U_j\ar@{->>}[r]\ \ar@{^{(}->}[d] &
\dfrac{U_j}{U_j\times_{Y_{i,j}} X_{i,j}}\ar@{^{(}->}[d] ^{m_{i,j}} \\
 X_{i,j} \ar@{^{(}->}[r] &Y_{i,j} \ar@{->>}[r]\ &\dfrac{Y_{i,j}}{X_{i,j}}  \\
}
$$
where the arrows $m_{i,j}$ are admissible monomorphisms.

\vspace{0.1cm}

Taking projective limits, we get an admissible monomorphism of $\Pro^a(\A)$:
$$
m_i=\limproj m_{i,j}:\limproj\dfrac{U_j}{U_j\times_{Y_{i,j}} X_{i,j}}\hrar\limproj\dfrac{Y_{i,j}}{X_{i,j}}
$$
and notice that 
$$
\limproj\dfrac{U_j}{U_j\times_{Y_{i,j}} X_{i,j}}=\dfrac{U}{U\times_{Y_i}X_i} \text{ \ \ \  and \ \ \  } \limproj\dfrac{Y_{i,j}}{X_{i,j}}=Z_i.
$$
Thus, for all $i$ we have an admissible monomorphism $m_i$ making  the diagram 
\begin{equation}\label{fullm_i}
\xymatrix{ 
U\times_{Y_i} X_i \ar@{^{(}->}[r]\ar@{^{(}->}[d] &U \ar@{->>}[r]\ar@{^{(}->}[d] &\dfrac{U}{U\times_{Y_i} X_i}\ar@{^{(}->}[d]^{m_i} \\
 X_i \ar@{^{(}->}[r] &Y_i\ar@{->>}[r] &Z_i  \\
}
\end{equation}
commutative.

\vspace{0.1cm}

We then repeat the same argument, this time taking inductive limits of the diagram \eqref{fullm_i}. When applying $\limi$ to the left square of \eqref{fullm_i} we get, as in Prop. \eqref{lift}, the commutative diagram \eqref{leftsquare}. When applied to the right square, we get an arrow
$$
m=\limi m_{i}:\limi\dfrac{U}{U\times_{Y_{i}} X_{i}}\hrar\limi\dfrac{Y_{i}}{X_{i}},
$$
i.e. an admissible monomorphism $m:\dfrac{U}{(U\cap X)}\hrar Z$, for which the diagram \eqref{lp} is commutative. This proves the first assertion of the Theorem. It is then left to prove that $[\dfrac{U}{U\cap X}\hrar Z]$ is in $\Gamma(Z)$.  Let us repeat the same argument, this time to the columns of Diagram \eqref{m_i}, that is, we take this time the quotients in the vertical direction. We obtain a commutative diagram:
 $$
\xymatrix{ 
U\cap X \ar@{^{(}->}[r]\ar@{^{(}->}[d] &U \ar@{^{(}->}[d] \\
X\ar@{->>}[d] \ar@{^{(}->}[r] &Y\ar@{->>}[d]  \\
\dfrac{X}{U\cap X} \ar@{^{(}->}[r]_{i} &\dfrac{Y}{U} \\
}
$$
where $i$ is an admissible monomorphism of $\Ind^a(\A)$. Then, we get a commutative diagram as follows:
$$
\xymatrix{ 
U\cap X \ar@{^{(}->}[r]\ar@{^{(}->}[d] &U \ar@{->>}[r]\ar@{^{(}->}[d] &\dfrac{U}{U\cap X}\ar@{^{(}->}[d]^{m}\\
X \ar@{^{(}->}[r]\ar@{->>}[d] &Y\ar@{->>}[d]\ar@{->>}[r] &Z\ar@{->>}[d]^{e}  \\
\dfrac{X}{U\cap X} \ar@{^{(}->}[r]_{i} &\dfrac{Y}{U}\ar@{->>}[r]_{j} & Q \\
}
$$
in which all the rows and columns are admissible short exact sequences, $Q$ is the common quotient and the bottom right square is a pushout square of admissible epimorphisms, as it can be seen by application of Proposition \eqref{selfdual}, or by dualizing Proposition \eqref{lift}.

\vspace{0.1cm}

Since $[U\cap X\hrar X]\in\Gamma(X)$ and $[U\hrar Y]\in\Gamma(Y)$, we get that $\dfrac{X}{U\cap X}$ and $\dfrac{Y}{U}$ are in $\Ind^a(\A)$. Hence their quotient $Q$  is in $\Ind^a(\A)$. But $Q$ is also the quotient $\dfrac{Z}{U/(U\cap X)}$, which is thus in $\Ind^a(\A)$. This shows that $\dfrac{U}{U\cap X}\in\Gamma(Z)$, and the proof of the Theorem is complete.
\end{proof}
\begin{cor}\label{squareofmonos}
 Let $\A$ be a partially abelian exact category and let $X_1\hrar X_2$ be an admissible monomorphism in $\limA$. Suppose we have a pullback diagram of admissible monomorphisms in  $\Pro^a(\A)$:
\begin{equation}
\xymatrix{
   U'_1 \ar@{^{(}->}[r]\ar@{^{(}->}[d] & U_1\ar@{^{(}->}[d] \\
U'_2\ar@{^{(}->}[r] & U_2
}
\end{equation}
where $U_1, U'_1\in\Gamma(X_1)$, $U_2, U'_2\in\Gamma(X_2)$. Then we have an induced commutative diagram
$$
\xymatrix{ 
U'_1 \ar@{^{(}->}[r]\ar@{^{(}->}[d] &U_1 \ar@{->>}[r]\ar@{^{(}->}[d] &U''_1\ar@{^{(}->}[d]^{}\\
U'_2 \ar@{^{(}->}[r]\ar@{->>}[d] &U_2\ar@{->>}[d]\ar@{->>}[r] &U''_2\ar@{->>}[d]^{}  \\
\dfrac{U'_2}{U'_1} \ar@{^{(}->}[r]_{} &\dfrac{U_2}{U_1}\ar@{->>}[r]_{} & \dfrac{U''_2}{U''_1} \\
}
$$
where all the rows and columns are admissible short exact sequences. In particular, the bottom row (and, symmetrically, the right column) is an admissible short exact sequence of $\A$.
\end{cor}

\section{The determinantal torsor on the Waldhausen space $S(lim \mathcal A)$}\label{sectiondim}
\subsection{The dimensional torsor}
\begin{definition} Let be: $\A$ an exact category, $G$ an abelian group. 
\vspace{0.1cm}

(a) A function $\chi: Ob\ \A\lrar G$ is called {\it dimensional theory on} $\A$ if, for any admissible short exact sequence of $\A, a'\hrar a\epi a''$, it is: 
$$
\chi(a)=\chi(a') +\chi(a'').
$$
\vspace{0.1cm}
(b) Let be $\chi$ a dimensional theory and $X\in\limA$. A {\it $\chi$-relative dimensional theory} is a map $d:\Gamma(X)\lrar G$ such that, for all admissible monomorphisms $U\hrar V$, between elements  $U, V\in\Gamma(X)$, we have:
\begin{equation}\label{dimx}
d(V)=d(U)+\chi\left (\dfrac{V}{U}\right ).
\end{equation}
(c) Given a dimensional theory $\chi$, we denote by $Dim_{\chi}(X)$ the set of all $\chi$-relative dimensional theories on $X$.
\end{definition}
As a consequence of (a), we have $\chi(0)=0$, and that whenever $a\xrar{\sim} a'$ it is $\chi(a)=\chi(a')$.  Moreover, let be $K_0(\A)$ the 
Grothendieck group of the exact category $\A$. From the universal property of $K_0(\A)$, the datum of a dimension theory 
$\chi: Ob\ \A\rar G$ is equivalent to the datum of a homomorphism $u_{\chi}:K_0(A)\rar G$.
\begin{lm} \label{uisov}
If $U\xrar{\sim} V$ in $\Gamma(X)$, then $d(U)=d(V)$, for all $d\in Dim_{\chi}(X)$.
\end{lm}
The proof is a simple consequence of the above mentioned fact  $\chi(0)=0$ applied to the admissible short exact sequence $U\xrar{\sim} V\rar 0$.
\begin{prop} 
Let $\A$ be an exact category satisfying the (AIC). Then, $Dim_{\chi}({X})$ is a $G$-torsor.
\end{prop}
\begin{proof} (Sketch.) We first define an action $G \times Dim_{\chi}({X})\stackrel{*}\lrar Dim_{\chi}({X})$ by letting, $\forall g\in G$, $d\in Dim_{\chi}({X})$, and $U\in \Gamma(X)$, $(g*d)(U):=g+d(U)$. It is immediate that this datum defines an action of $G$ onto $Dim_{\chi}({X})$. We prove that it is free and transitive. To see this, let be $d, \ \ d'\in Dim_{\chi}(X)$. Let's fix $U\in\Gamma(X)$. Then, write $g:=d(U)-d'(U).$ It's enough to prove that this $g$ does not depend on $U$.  The argument works as follows: let be $U\hrar U'$ in $\Gamma(X)$ We have: $d(U')=d(U)+\chi\left (\dfrac{U'}{U}\right )$; $d'(U')=d'(U)+\chi\left (\dfrac{U'}{U}\right )$. Thus, since $G$ is abelian, it follows $d(U')-d'(U')=d(U)-d'(U)=g$. When $U'$ is any element of $\Gamma(X)$, Theorem \eqref{limintersections} shows that $U\cap U'$ is in $\Gamma(X)$. The consideration of the diagram \eqref{intgr} proves that it is $d(U')-d(U)=g$ also in this case.
\end{proof}

\subsection{The universal dimensional torsor}\label{sectionunivdim}

It is possible to introduce a dimensional torsor which is ``universal" in the sense that it depends only on the category $\A$ via the Grothendieck group $K_0(\A)$. This dimensional torsor will be denoted by $\Dim(X)$.

\begin{definition} 
Let be $\psi:Ob\ \A\rar K_0(\A)$ the dimension theory sending each $a\in Ob\ \A$ into its class $[a]\in K_0(\A)$. We shall call this 
function $\psi$ the {\it universal dimension theory} on $\A$. We then define
$$
\Dim(X):=\Dim_{\psi}(X).
$$
\end{definition}
Thus, $\Dim(X)$ is the $K_0(\A)$-torsor associated with the identity on $K_0(\A)$. 

\vspace{0.1cm}

If $\chi\colon Ob\ \A\to G$ is any other dimension theory, and $u_{\chi}$ the corresponding group morphism $K_0(\A)\rar G$, we have the following
\begin{prop}\label{univdim}
$u_{\chi_*}(\Dim(X))=\Dim_{\chi}(X)$.
\end{prop}

\begin{example}
{\it The Kapranov Dimensional torsor} $\Dim(V)$. \\
Let be $\A=\vect$. We have $\limA=\T$, the category of Tate spaces. Let be $V\in\T$ a Tate space. Since $K_0(\A)=\Zee$, $\Dim
(V)$ is a 
$\Zee$-torsor. This torsor is the {\it Kapranov dimensional torsor} associated to a Tate space $V$, defined by Kapranov in \cite{ka}.
\end{example}
\subsection{Dimensional torsors form a symmetric determinantal theory}\label{sectionmult}

We now  study the behavior of the dimensional torsor with respect to admissible short exact sequences of $\limA$, where $\A$ is a partially abelian exact category. 

Let $\A$ be a partially abelian exact category and $\chi$ be a dimensional theory on $\A$ with values in an abelian group $G$. Consider, in $\limA$, any admissible short exact sequence:
$$
X'\hrar X\epi X''.
$$
Let $[U\hrar X]\in\Gamma(X)$. Since $\A$ is partially abelian exact, from Theorem \eqref{projlift} we get the admissible short exact sequence of $\Pro^a(\A)$:
$$
U\cap X'\hrar U\epi\frac{U}{U\cap X'}
$$
where $[U\cap X'\hrar X']\in\Gamma(X')$  and $\left [\dfrac{U}{U\cap X'}\hrar X''\right ]\in\Gamma(X'')$.

\vspace{0.3cm}

Next, let be given $d'\in\Dim_{\chi}(X)$ and $d''\in\Dim_{\chi}(X'')$. Define:
\begin{equation}\label{d}
d(U):=d'(U\cap X')+d''\left (\dfrac{U}{U\cap X'}\right )
\end{equation}
for all $[U\hrar X]\in\Gamma(X)$.

\begin{theorem}\label{tensordim}
(1) The map $d:\Gamma(X)\lrar G$ defined in \eqref{d} is a $\chi$-relative dimensional theory on $X$.
\par 
(2) The induced map $\mu:\Dim_{\chi}(X')\times\Dim_{\chi}(X'')\rar\Dim(X)$ given by $\mu(d', d'')=d$ descends to a (iso-)morphism of $G$-torsors:
$$
\mu_{X',X,X''}:\Dim_{\chi}(X')\otimes\Dim_{\chi}(X'') \stackrel{}\lrar\Dim_{\chi}(X)
$$
for which the pair $(\Dim_{\chi}(X), \mu)_{X\in\limA}$ is a symmetric determinantal theory on $\limA$, with values in the Picard category $\Tors(G)$.
\end{theorem}
\begin{proof} 
\par
(1) Let be $U_1\hrar U_2\hrar X$ with $[U_1\hrar X], \  [U_2\hrar X]\in\Gamma(X)$.

\vspace{0.1cm}

We have the following relations: 
$$
d(U_2):=d'(U_2\cap X')+d''\left (\dfrac{U_2}{U_2\cap X'}\right )
$$
$$
d(U_1):=d'(U_1\cap X')+d''\left (\dfrac{U_1}{U_1\cap X'}\right )
$$
$$
d'(U_2\cap X'):=d'(U_1\cap X')+\chi\left (\dfrac{U_2\cap X'}{U_1\cap X'}\right )
$$
and
$$
d''\left (\dfrac{U_2}{U_2\cap X'}\right ):=d''\left (\dfrac{U_1}{U_1\cap X'}\right )+
\chi\left (\dfrac{\dfrac{U_2}{U_2\cap X'}}{\dfrac{U_1}{U_1\cap X'}}\right ).
$$
It's enough to prove $d(U_2)=d(U_1)+\chi(U_2/U_1)$. This results from the above relations thanks to the commutativity of $G$ and because the sequence 
$$
\dfrac{U_2\cap X'}{U_1\cap X'}\hrar \dfrac{U_2}{U_1}\epi\dfrac{\dfrac{U_2}{U_2\cap X'}}{\dfrac{U_1}{U_1\cap X'}}
$$
is an admissible short exact sequence of $\A$, from Theorem \eqref{projlift}. Thus, since $\chi$ is defined on $K_0(\A)$, we have 
$$
\chi\left (\dfrac{U_2}{U_1}\right )=\chi\left (\dfrac{U_2\cap X'}{U_1\cap X'}\right )+\chi\left (\dfrac{\dfrac{U_2}{U_2\cap X'}}{\dfrac{U_1}{U_1\cap X'}}\right ).
$$
Substituting in the expression obtained for $d(U_2)-d(U_1)$, we get $\chi(U_2/U_1)$, as claimed.

\vspace{0.1cm}

(2) To check that $\mu$ descends to a (iso)morphism of torsors, it is enough to check that $\forall g\in G$, it is $\mu(gd', d'')=\mu(d', gd'')$. But this is immediate from the definition of $\mu$ and the commutativity  of $G$. 

\vspace{0.1cm}

In order to prove that $(\Dim_{\chi}(X), \mu)_{X\in\limA}$ is a determinantal theory we need to show that the isomorphisms $\mu$ are natural with respect to isomorphisms of admissible short exact sequences of $\limA$, and that diagram \eqref{mult}
 commutes for $h(X)=\Dim(X), a_i=X_i$ for $i=1,2,3$ and $\lambda=\mu$. We shall need a topological lemma about the Grassmannians.

\vspace{0.1cm}

Let be given the following diagram in $\limA$, where the horizontal arrows are admissible monomorphisms and the vertical ones the corresponding cokernels: 
\begin{equation}\label{mult2}
\xymatrix{ 
X_1  \ar@{^{(}->}[r]& X_2\ar@{->>}[d]\ar@{^{(}->}[r] & X_3\ar@{->>}[d] \\
& \dfrac{X_2}{X_1} \ar@{^{(}->}[r] &\dfrac{X_3}{X_1}\ar@{->>}[d]\\
&&\dfrac{X_3}{X_2}
}
\end{equation}
We are given 3 dimension theories, $d_1, d_{21}, d_{32}$ resp. on $X_1$, $\dfrac{X_2}{X_1}$ and $\dfrac{X_3}{X_2}$. The commutativity of diagram \eqref{mult} is equivalent to the equality
$$
\mu(\mu(d_1, d_{21}), d_{32})=\mu(d_1, \mu(d_{21}, d_{32})),
$$
as dimension theories on $X_3$.

\vspace{0.1cm}

Let thus $[U\hrar X_3]\in\Gamma(X_3)$ be given, and let us construct first $\mu(\mu(d_1, d_{21}), d_{32})$, by applying \eqref{d}.  

\vspace{0.1cm}

We first lift $U$ along $X_2\hrar X_3$ to $U_2=U\cap X_2\in\Gamma(X_2)$. We then project $U_2$ along $X_3\epi\dfrac{X_3}{X_2}$ to an element $U_{32}=\dfrac{U}{U_2}\in\Gamma\left (\dfrac{X_3}{X_2}\right )$. The element $U_2$ is then lifted along $X_1\hrar X_2$ to $U_1=U\cap X_1\in\Gamma(X_1)$, and then projected along $X_1\epi\dfrac{X_2}{X_1}$ to the element 
$$
U_{21}=\dfrac{U_2}{U_1}=\dfrac{U\cap X_2}{U\cap X_1}\in\Gamma\left (\dfrac{X_2}{X_1}\right ).
$$
We thus can write:
$$
\mu(\mu(d_1, d_{21}), d_{32})=d_1(U_1)+d_{21}(U_{21})+d_{32}(U_{32}).
$$
We similarly construct $\mu(d_1, \mu(d_{21}, d_{32}))$, as follows: we first lift $U$ to the same $U_1\in\Gamma(X_1)$, since pullbacks are unique up to a unique isomorphism. We then project $U$ along $X_3\epi\dfrac{X_3}{X_1}$, to obtain $\dfrac{U}{U_1}\in\Gamma\left (\dfrac{X_3}{X_1}\right )$. This element is then lifted along $\dfrac{X_2}{X_1}\hrar \dfrac{X_3}{X_1}$ to the element 
$$
U'_{21}=\dfrac{U}{U_1}\cap\dfrac{X_2}{X_1}\in\Gamma\left (\dfrac{X_2}{X_1}\right ),
$$
then projected along $\dfrac{X_3}{X_1}\hrar \dfrac{X_3}{X_2}$ to the element 
$U'_{32}=\dfrac{U/U_1}{U'_{21}}$ in $\Gamma\left (\dfrac{X_3}{X_2}\right )$. We thus have:
$$
\mu(d_1, \mu(d_{21}, d_{32}))(U)=d_1(U_1)+d_{21}(U'_{21})+d_{32}(U'_{32}).
$$
The proof of (4) is then an immediate consequence of the following
\begin{lm}\label{fundlemma} 
In the above situation, we have equalities $U_{21}=U'_{21}$ in $\Gamma\left (\dfrac{X_2}{X_1}\right )$ and $U_{32}=U'_{32}$ in $\Gamma\left (\dfrac{X_3}{X_2}\right )$.
\end{lm}
Both equalities are general properties holding in any abelian category. The first equality, in set-theoretical terms, reads:
$$
\dfrac{U\cap X_2}{U\cap X_1}=\ker\left \{\dfrac{U}{U\cap X_1}\lrar\dfrac{X_3}{X_2}\right \}
$$
as subobjects in 
$$
\dfrac{X_2}{X_1}=\ker\left\{\dfrac{X_3}{X_1}\epi\dfrac{X_3}{X_2}\right\}.
$$
The second equality is a consequence of the first, since $\dfrac{U/U_1}{U_2/U_1}=\dfrac{U}{U_2}$,
and the lemma is proved. It remains to check the symmetry of $\Dim_{\chi}(X)$. This is easily done directly using proposition \eqref{symmdet}. The proof of Theorem \eqref{tensordim} is now complete.
\end{proof}
\subsection{Cohomological interpretation of Dim(X) in terms of the Waldhausen space of lim $\mathcal A$} \label{sectionwaldhausen}
We refer to multiplicative torsors of degree $n$ over the bisimplicial set determined by $\SA$ simply as multiplicative torsors of degree $n$ over $S(\A)$. 

\vspace{0.1cm}

For $i=0$ the relation $\pi_{i+1}(S(\A))=K_i(\A)$ gives $\pi_1(S(\A))=H_1(S(\A), \Zee)=K_0(\A)$, the Grothendieck group of the category $\A$. So the universal dimensional theory on $\A$ gives rise to a class $\zeta\in H^1(S(\A), K_0(\A))$. 
\begin{example}
It is immediate from the definitions that a 0-multiplicative $G$-torsor on $S(\A)$ is a dimensional theory, and that a 1-multiplicative $G$-torsor on $S(\A)$ is a determinantal theory on $\A$.
\end{example}
It is then possible to re-interpret theorem \eqref{tensordim} in terms of the Waldhausen space of $\limA$, as follows:
\begin{theorem}\label{dimxismult}
Let be $\A$ a partially abelian exact category. Let be $G$ an abelian group and $\chi:K_0(\A)\rar G$ a homomorphism. Therefore, the collection  $\{\Dim(X), X\in\limA\}$ is a multiplicative $G$-torsor on $S(\limA)$.
\end{theorem}
\begin{cor}\label{cortorsmult}
The class $\zeta$ in $H^1(S(\A), K_0(\A))$ gives rise to a cohomology class in \linebreak $H^2(S(\limA), K_0(\A))$.
\end{cor}
This result can be interpreted as a ``first step delooping" between the first cohomology of $S(\A)$ and the second cohomology of $S(\limA)$.
\begin{proof}
The theorem is a restatement of Theorem \eqref{tensordim}: the claims proved there are equivalent to the statement that $\{\Dim(X); X\in\limA\}$ is a multiplicative torsor. The corollary follows when we let $G=K_0(\A)$. Then, from Theorem \eqref{torsorcohomol}, the induced multiplicative $K_0(\A)-$torsor $\Dim(X)$ represents an element of $H^2(S(\limA), K_0(\A))$.
\end{proof}

\subsection{The determinantal torsor $\mathcal D_h(X)$} \label{sectiondet}
Given a Tate space $V$, we generalize the construction of the determinantal gerbe $\Det(V)$ (cf. \cite{ka}) in two directions: in the first place we consider any generalized Tate space $X$, provided that the exact base category $\A$ is partially abelian exact; secondly, instead of a gerbe over an abelian group, we shall produce a torsor $\D$ over a symmetric Picard category $\Pp$, which gives rise to the gerbe $\Det$ when we restrict to $\pi_1(\Pp)$, and to the torsor $\Dim$ when restricted to $\pi_0(\Pp)$.

\vspace{0.1cm}

Let $\A$ be an exact category and $X$ an object of $\limA$. Let be $(h, \lambda)$ a determinantal theory on $\A$ with values in a symmetric Picard category $\Pp$.

\begin{definition}
 A  {\it h-relative determinantal theory $\Delta$ on X} is the datum consisting of a pair $(\Delta, \delta)$, where $\Delta$ is a function $\Delta:\Gamma(X)\lrar\Ob\Pp$, such that:
 \begin{enumerate}
\item{For all admissible monomorphism $U\hrar V$ in $\Gamma(X)$, an isomorphism 
$$
\delta_{U,V}\colon\Delta(U)\otimes h\left (\dfrac{V}{U}\right )\xrar{\sim}\Delta(V),
$$
natural with respect to isomorphisms of admissible short exact sequences $U\hrar V\epi\dfrac{V}{U}$.}
\item{For all filtration of length 2 of admissible monomorphisms in $\Gamma(X)$, $U_1\hrar U_2\hrar U_3$, a commutative diagram
\begin{equation}\label{delta}
\xymatrix{ 
\Delta(U_1)\otimes h\left (\dfrac{U_{2}}{U_{1}}\right )\otimes h\left (\dfrac{U_{3}}{U_{2}}\right )\ar[d]_{\delta_{U_1,U_2}\otimes 1} \ar[rr]^-{1\otimes\lambda} &&\Delta(U_1)\otimes h\left (\dfrac{U_{3}}{U_{1}}\right )\ar[d]^{\delta_{U_1,U_3}}\\
\Delta(U_2)\otimes h\left (\dfrac{U_{3}}{U_{2}}\right ) \ar[rr]_-{\delta_{U_2,U_3}}&& \Delta(U_3) \\
}
\end{equation}}
where, as before, we have omitted the associator for simplicity.
\end{enumerate}
\end{definition}

A {\it morphism of h-relative determinantal theories} $f\colon (\Delta, \delta)\to(\Delta',\delta')$ is a collection of isomorphisms of $\Pp$, $
\{f_U\colon\Delta(U)\to\Delta '(U)\}_{U\in\Gamma(X)}$, such that, for $U\hrar V$ in $\Gamma(X)$, the diagram 
$$
\xymatrix{
\Delta(U)\otimes h\left (\dfrac{V}{U}\right )\ar[r]^-{\delta_{U,V}}\ar[d]_{f_{U}\otimes 1} &\Delta(V)\ar[d]^{f_V} \\
\Delta'(U)\otimes h\left (\dfrac{V}{U}\right )\ar[r]^-{\delta'_{U,V}} & \Delta'(V)
}
$$
commutes. It is clear that any such morphism is invertible, hence an isomorphism.

\begin{definition}
Let be $\Pp$ a Picard category, and let $X$ be as before an object in $\limA$. We denote by $\D_h(X,\Pp)$, or simply $\D_h(X)$
if no confusion arises, the category (groupoid) whose  objects are $h$-relative determinantal theories on X with values in $\Pp$ and 
morphisms are the morphisms of determinantal theories.
\end{definition}

\begin{theorem}\label{detxtorsor}
If the exact category $\A$ satisfies the (AIC), $\D_h(X)$ is a $\Pp$-torsor.
\end{theorem}

\begin{proof}

The action of $\Pp$ onto $\D_h(X)$ is defined as follows, for all objects $a\in\Pp$, $(\Delta, \delta)\in\D_h(X)$, and $U$ in $\Gamma(X)
$:
\begin{align*}
\Pp\times\D_h(X)&\xrightarrow{\otimes} \D_h(X) \\
(a, \Delta)(U)&\mapsto a\otimes\Delta(U)
\end{align*}
and extended to the arrows of $\Pp$ and $\Det_h(X)$ in the obvious way. 

\vspace{0.1cm}

Let us fix $(\Delta_0,\delta_0)$ and consider the induced functor:
\begin{align}
\Pp&\xrightarrow{-\otimes\Delta_0} \D_h(X) \\
b&\mapsto b\otimes \Delta_0
\end{align}
We sketch the proof that the  functor is an equivalence of categories. We first show that it is essentially surjective. 

\vspace{0.1cm}

Let $(\Delta, \delta)\in\Ob\D_h(X)$ be given. We shall prove the existence of an object $a\in\Pp$ and an isomorphism of determinantal theories:
$$
\Delta\xrar{\sim}a\otimes\Delta_0.
$$
Let $(\Delta,\delta)\in\Det_h(X)$ be given. Let us choose $[U\hrar X]\in\Gamma(X)$. In $\Pp$, consider the following isomorphism, naturally defined in $\Pp$:
$$
\Delta(U)\xrar{\sim}\Delta(U)\otimes 1\xrar{\sim}\Delta(U)\otimes(\Delta_0(U)^*\otimes\Delta_0(U))\xrar{\alpha}(\Delta(U)\otimes\Delta_0(U)^*)\otimes\Delta_0(U)
$$
where the first is the isomorphism given by 1 as a null object of $\Pp$ and the second is the isomorphism of duality for objects of $\Pp$.
We let $a:=\Delta(U)\otimes\Delta_0(U)^*$, and we write the above composition as
$$
f_U\colon\Delta(U)\xrar{\sim} a\otimes\Delta_0(U).
$$
We have to show the following: for all $W_1\hrar W_2$ in $\Gamma(X)$, there are isomorphisms 
$f_{W_1}\colon\Delta(W_1)\xrar{\sim}a\otimes\Delta_0(W_1)$ and $f_{W_2}\colon\Delta(W_2)\xrar{\sim}a\otimes\Delta_0(W_2)$, for which the diagram 
\begin{equation}\label{W1W2}
\xymatrix{
\Delta(W_1)\otimes h\left (\dfrac{W_2}{W_1}\right )\ar[rr]^-{\delta_{W_1,W_2}}\ar[d]_{f_{W_1}\otimes 1} &&\Delta(W_2)\ar[d]^{f_{W_2}} \\
a\otimes\Delta_0(W_1)\otimes h\left (\dfrac{W_2}{W_1}\right )\ar[rr]^-{1\otimes\delta_{0_{W_1,W_2}}} && a\otimes\Delta_0(W_2)
}
\end{equation}
is commutative.

\vspace{0.1cm}

We start by defining $f_V$ for alll $V$ with $U\hrar V$. In this case, $f_V$ is defined as the dotted arrow of the diagram below, i.e. as the composition of the isomorphisms represented by full arrows as the already defined morphisms:
\begin{equation}\label{UV}
\xymatrix{
\Delta(U)\otimes h\left (\dfrac{V}{U}\right )\ar[rr]^-{\delta_{U,V}}\ar[d]_{f_{U}\otimes 1} &&\Delta(V)\ar@{-->}[d]^{f_V} \\
a\otimes \Delta_0(U)\otimes h\left (\dfrac{V}{U}\right )\ar[rr]^-{1\otimes\delta_{0_{U,V}}} && a\otimes\Delta_0(V)
}
\end{equation}
Similarly one defines $f_V$ if $V\hrar U$.

\vspace{0.1cm}

Next, let be $W\in\Gamma(X)$. To define $f_W$, we consider the following diagram in $\limA$, whose existence follows from Theorem \eqref{limintersections}:
\begin{equation}\label{WU}
 \xymatrix{
U\cap W\ar@{^{(}->}[d]\ar@{^{(}->}[r]   & W\ar@{^{(}->}[d] \\
U \ar@{^{(}->}[r] & X
}
\end{equation}
with $[W\cap U]\in\Gamma(X)$.  We use  diagram \eqref{UV} applied to $V=W\cap U$ and $U=U$. This defines $f_{W\cap U}$. 

\vspace{0.1cm}

From $f_{W\cap U}$ we can define $f_W\colon\Delta(W)\rar a\otimes\Delta_0(W)$ using again diagram \eqref{UV}, where now it is $U=U\cap W$ and $V=W$. This defines $f_W$ for all $W\in\Gamma(X)$.

\vspace{0.1cm}

It therefore remains to prove that for all $W_1\hrar W_2$ in $\Gamma(X)$, diagram \eqref{W1W2} commutes, where $f_{W_1}$ and $f_{W_2}$ have been constructed according with the above procedure, for $W=W_1$ and $W=W_2$.

\vspace{0.1cm}

Let us first consider two elements $V_1, V_2\in\Gamma(X)$, such that $U\hrar V_1\hrar V_2$. From the isomorphism
$$
\lambda\colon h\left (\dfrac{V_1}{U}\right )\otimes h\left (\dfrac{V_2}{V_1}\right )\xrar{\sim} h\left (\dfrac{V_2}{U}\right )
$$
we obtain the following commutative diagram:
\begin{equation}\label{V1V2}
\xymatrix{
\Delta(V_1)\otimes h\left (\dfrac{V_2}{V_1}\right )\ar[r]^-{\delta^{-1}\otimes 1}\ar[d]_{f_{V_1}\otimes 1} &
\Delta(U)\otimes h\left (\dfrac{V_1}{U}\right )\otimes h\left (\dfrac{V_2}{V_1}\right )\ar[d]^{f_{U}\otimes 1}\ar[r]^-{1\otimes\lambda}&\Delta(U)\otimes h\left (\dfrac{V_2}{U}\right )\ar[d]^{f_U\otimes 1}\ar[r]^-{\delta}& \Delta(V_2)\ar[d]^{f_{V_2}} \\
a\otimes\Delta_0(V_1)\otimes h\left (\dfrac{V_2}{V_1}\right )\ar[r]^-{1\otimes\delta^{-1}_{0}\otimes 1} & a\otimes\Delta_0(U)
\otimes h\left (\dfrac{V_1}{U}\right )\otimes h\left (\dfrac{V_2}{V_1}\right )\ar[r]^-{1\otimes\lambda} & a\otimes\Delta_0(U)\otimes h\left (\dfrac{V_2}{U}\right )\ar[r]^-{\delta} & a\otimes\Delta_0(V_2)
}
\end{equation}
this diagram allows to express $f_{V_2}$ in terms of $f_{V_1}$. We have a similar diagram when $V_1\hrar V_2\hrar U$. 

\vspace{0.1cm}

Let now be $W_1\hrar W_2$. Then, $U\cap W_1\hrar U\cap W_2\hrar U$ and $U\cap W_1\hrar W_1\hrar W_2$ are admissible filtrations in $\Gamma(X)$.

\vspace{0.2cm}

From diagram \eqref{V1V2}, applied to the first filtration, we obtain a diagram of type \eqref{UV}, with $U\cap W_1$ as $U$ and $U\cap W_2$ as $V$. We compose this diagram with the diagram defining $f_{W_2}$. We get
$$
\xymatrix{
\Delta(U\cap W_1)\otimes h\left (\dfrac{U\cap W_2}{U\cap W_1}\right )\otimes h\left (\dfrac{W_2}{U\cap W_2}\right )\ar[r]
\ar[d]_{f_{U\cap W_1}\otimes 1} &\Delta(W_2)\ar[d]^{f_{W_2}} \\
a\otimes\Delta_0(U\cap W_1)\otimes h\left (\dfrac{U\cap W_2}{U\cap W_1}\right )\otimes h\left (\dfrac{W_2}{U\cap W_2}\right )\ar[r] & a\otimes\Delta_0(W_2)
}
$$
On the other hand, we have isomorphisms:
$$
h\left (\dfrac{U\cap W_2}{U\cap W_1}\right )\otimes h\left (\dfrac{W_2}{U\cap W_2}\right )\xrar{\sim} h\left (\dfrac{W_2}{U\cap W_1}\right )
$$
from the first filtration, and 
$$
 h\left (\dfrac{W_2}{U\cap W_1}\right )\xrar{\sim} h\left (\dfrac{W_1}{U\cap W_1}\right )\otimes  h\left (\dfrac{W_2}{W_1}\right )
 $$
 from the second. Thus, 
 $$
 h\left (\dfrac{U\cap W_2}{U\cap W_1}\right )\otimes h\left (\dfrac{W_2}{U\cap W_2}\right )\xrar{\sim}
 h\left (\dfrac{W_1}{U\cap W_1}\right )\otimes  h\left (\dfrac{W_2}{W_1}\right ).
 $$
 The above diagram can thus be rewritten as:
 $$
\xymatrix{
\Delta(U\cap W_1)\otimes h\left (\dfrac{W_1}{U\cap W_1}\right )\otimes h\left (\dfrac{W_2}{W_1}\right )\ar[r]
\ar[d]_{f_{U\cap W_1}\otimes 1} &\Delta(W_2)\ar[d]^{f_{W_2}} \\
a\otimes\Delta_0(U\cap W_1)\otimes h\left (\dfrac{W_1}{U\cap W_1}\right )\otimes h\left (\dfrac{W_2}{W_1}\right )\ar[r] & a\otimes\Delta_0(W_2).
}
$$
Composing this diagram with the diagram \eqref{UV} defining $f_{W_1}$, we finally get diagram \eqref{W1W2}, thus proving that the functor $\Pp\rar\D_h(X)$ is essentially surjective. 

\vspace{0.1cm}

We sketch the proof that the functor is full. This amounts to show that for all object of $\Pp$, the map
\begin{align}
\Hom_{\Pp}(a,b)&\rar\Hom_{\D_h(X)}(a\otimes\Delta_0,b\otimes\Delta_0) \\
h&\mapsto h\otimes 1_{\Delta_0}
\end{align}
is surjective. We shall consider only the case $a=b=1$, since the general case is treated with the obvious modifications.

\vspace{0.1cm}

Let be $f\in\Aut(1\otimes\Delta_0)\xrar{\sim}\Aut(\Delta_0)$. Let us choose $U\in\Gamma(X)$ and 
$f_U\colon\Delta_0(U)\xrar{\sim}\Delta_0(U)$ be given. There is a unique $g:1\rar 1$ making the following diagram commute:
$$
\xymatrix{
\Delta_0(U)\otimes\Delta_0(U)^{-1}\ar[d]\ar[rr]^-{f_0\otimes 1_{\Delta_0(U)^{-1}}}&&\Delta_0(U)\otimes\Delta_0(U)^{-1}\ar[d] \\
1\ar@{-->}[rr]_{g} && 1 .
}
$$ 
For this $g$ we have $g\otimes 1_{\Delta_0(U)}=f_U\colon1\otimes\Delta_0(U)\rar 1\otimes\Delta_0(U)$. It is sufficient to prove that for all $V\in\Gamma(X)$ the arrow $g\otimes 1_{\Delta_0(V)}$ coincides with $f_V$. This will imply that $f=g\otimes 1_{\Delta_0}$, and hence that the functor is full.

\vspace{0.1cm}

Let be $V$ with $U\hrar V$. In this case $f_V$ is the unique arrow making the diagram
$$
\xymatrix{
1\otimes\Delta_0(U)\otimes h\left (\dfrac{V}{U}\right )\ar[rr]^-{1\otimes\delta_{0}}\ar[d]_{f_{U}\otimes 1} &&1\otimes\Delta_0(V)\ar[d]^{f_V} \\
1\otimes \Delta_0(U)\otimes h\left (\dfrac{V}{U}\right )\ar[rr]^-{1\otimes\delta_{0}} && 1\otimes\Delta_0(V)
}
$$
commutative. From the diagram for the identity on $(\Delta_0,\delta_0)$, applying $g:1\rar 1$ and using the bifunctoriality olf $\otimes$, we get the following:
$$
\xymatrix{
1\otimes\Delta_0(U)\otimes h\left (\dfrac{V}{U}\right )\ar[rr]^-{1\otimes\delta_{0}}\ar[d]_{g\otimes 1_{\Delta_0(U)}\otimes 1} &&1\otimes\Delta_0(V)\ar[d]^{g\otimes 1_{\Delta_0(V)}} \\
1\otimes \Delta_0(U)\otimes h\left (\dfrac{V}{U}\right )\ar[rr]^-{1\otimes\delta_{0}} && 1\otimes\Delta_0(V)
}
$$
but $g\otimes 1_{\Delta_0(U)}=f_U$. It follows that $g\otimes 1_{\Delta_0(V)}$ coincides with $f_V$, as claimed. The proof for the general case of an element $W\in\Gamma(X)$ follows the same pattern as the proof that $-\otimes\Delta_0$ is essentially surjective. Thus the functor is full. Injectivity of the map is obvious, so the functor is also faithful and then an equivalence.
\end{proof}
\subsection{Examples: the gerbe of determinantal theories Det(V)}
As a corollary of Proposition \eqref{pi1gerbe} and of Theorem \eqref{detxtorsor}, we have the following
\begin{prop}
Let $\A$ be an exact category satisfying the (AIC), and let be $h$ a determinantal theory on $\A$ with values in a Picard category $\Pp$, and choose an object $(\Delta,\delta)\in\D_h(X)$. Then $\D_h(X)_{(\Delta,\delta)}$ is a $\pi_1(\Pp)$-gerbe. 
\end{prop}
With the same step-by-step method used to prove in Theorem \eqref{detxtorsor} the existence of an isomorphism of $h$-relative determinantal theories, we can prove the following
\begin{lm}\label{connected}
If $\Pp$ is a connected category, then $\D_h(X)$ is a connected groupoid.
\end{lm}
Thus, if $\Pp$ is connected, for all $(\Delta, \delta)$ we have $\D_h(X)_{(\Delta, \delta)}=\D_h(X)$, which is then a $\pi_1(\Pp)$-gerbe.
\vspace{0.1cm}

\begin{examples}\label{detexamples}
\begin{enumerate}
\item{ {\it The Kapranov gerbe of determinantal theories} $\Det(V)$.
Let be $k$ a field, $\A=\vect$, and let $\Pp$ be the category $\onevect$ of 1-dimensional vector spaces over $k$. The category $\Pp$ is obviously a connected Picard category, with $\pi_1(\Pp)=k^*$. Let $h$ be the determinantal theory on $\A$ which associates to each finite dimensional space its determinantal space (as described in the example in sect. \eqref{detspace}); finally, let be $V\in\dlim\vect$ a Tate space. Define $\Det(V):=\D_h(V)$.  From Lemma \eqref{connected}, $\Det(V)$ is connected and thus it is a $k^*$-gerbe. It is called the {\it gerbe of determinantal theories} of the Tate space $V$, and it was introduced by Kapranov in his paper \cite{ka}.}
\vspace{0.1cm}

\item{ Let $\A$ be an exact category satisfying the (AIC), $G$ an abelian group and let be $\Pp=\Tors(G)$. Let be $h$ a determinantal theory on $\A$ with values on $\Pp$. The category $\Pp$ is a connected Picard category, so we have $\D_h(X)_{(\Delta,\delta)}=\D_h(X)$, for each such $h$ and any determinantal theory $(\Delta, \delta)$. Since $\pi_1(\Pp)=G$,  $\D_h(X)$ is a $G$-gerbe. In this case we shall employ also the notation $\Det_h(X)$, to emphasize that this is really the Kapranov $G$-gerbe of determinantal theories of a generalized Tate space $X$.}
\end{enumerate}
\end{examples}
\subsection{The universal $\mathcal D(X)$}
In analogy with the case of the dimensional torsor $\Dim(X)$, it is possible to define a ``universal" determinantal torsor $\D(X)$ over Picard categories.

\begin{definition}
Let be $\A$ an exact category satisfying the (AIC) and $X$ an object of $\limA$. The {\it universal determinantal torsor} is the $V(\A)$-torsor
$$
\D(X)\colon=\D_{h^u}(X, V(\A))
$$
associated with the symmetric universal determinantal theory on $\A$, $h^u$, defined in sect.\eqref {detspace}.
\end{definition}
\begin{remark} 
It is possible to characterize $\D(X)$ by an appropriate 2-categorical universal property. We postpone the precise statement and the discussion of this topic to a later paper.
\end{remark}
\begin{example} Let $\A=\vect$, and let be $V\in\limA=\T$, a Tate space. In this case, the category of virtual objects $\Pp=V(\A)$ has 
$\pi_0(\Pp)=\Zee$, and thus it is not connected. Since $\pi_1(\Pp)=k^*$, the universal determinantal $V(\A)$-torsor $\D_{h^u}(V)$ is a non-connected groupoid. Each of its connected components $\D(V)_{(\Delta, \delta)}$, is a $k^*$-gerbe, and all of these components compose a set (indexed by $\Zee$) of copies of the Kapranov $k^*$-gerbe $\Det(V)$.
\end{example}

\subsection{Cohomological interpretation of Det(X) in terms of the Waldhausen space of lim $\mathcal A$}

\begin{theorem}
Let be $G$ an abelian group and $\A$ a partially abelian exact category. Let be 
$(h,\lambda)\in\Det_{\sigma}(\A, \Tors(G))$ a symmetric determinantal theory on $\A$ with values in the Picard category of $G$-torsors. Then, the collection 
$\{\Det_h(X), X\in\limA\}$, defined in \eqref{detexamples}, is a multiplicative $G$-gerbe of degree 1 on 
$S(\limA)$.
\end{theorem}
\begin{proof}(Sketch.) The proof of this theorem is  similar to, although considerably longer than, the proof of Theorem \eqref{dimxismult}. We emphasize only the most salient points. As already noticed, $(h,\lambda)$ can be interpreted as a symmetric multiplicative torsor on $\A$. The core of the proof consists in showing the existence of an equivalence of $G$-gerbes $\mu:\Det_h(X')\otimes\Det_h(X'')\rar\Det_h(X)$ for an admissible short exact 
sequence $X'\hrar X\epi X''$ in $\limA$.

\vspace{0.1cm}

For an admissible filtration $U_1\hrar U_2\hrar U_3$ in $\Gamma(X)$, consider the induced commutative diagram
$$
\xymatrix{ 
U'_1 \ar@{^{(}->}[r]\ar@{^{(}->}[d] &U_1 \ar@{->>}[r]\ar@{^{(}->}[d] &U''_1\ar@{^{(}->}[d]^{}\\
U'_2 \ar@{^{(}->}[r]\ar@{^{(}->}[d] &U_2\ar@{^{(}->}[d]\ar@{->>}[r] &U''_2\ar@{^{(}->}[d]^{}  \\
U'_3 \ar@{^{(}->}[r]\ar@{^{(}->}[d] &U_3\ar@{^{(}->}[d]\ar@{->>}[r] &U''_3\ar@{^{(}->}[d]^{}  \\
X' \ar@{^{(}->}[r] &X\ar@{->>}[r] &X'' 
}
$$
whose left squares are pullbacks, the horizontal rows admissible short exact sequences; then from theorem \eqref{projlift} 
$U''_1\hrar U''_2\hrar U''_3$ is an admissible filtration in $\Gamma(X'')$. Since $\A$ is partially abelian, we can use corollary \eqref{squareofmonos}, and thus we obtain an induced commutative diagram
\begin{equation}\label{quotients}
\xymatrix{ 
\dfrac{U'_3}{U'_2}\ar@{^{(}->}[r] &\dfrac{U_3}{U_2} \ar@{->>}[r] &\dfrac{U''_3}{U''_2}\\
\dfrac{U'_3}{U'_1} \ar@{^{(}->}[r]\ar@{->>}[u] &\dfrac{U_3}{U_1}\ar@{->>}[u]\ar@{->>}[r] &\dfrac{U''_3}{U''_1}\ar@{->>}[u]^{}  \\
\dfrac{U'_2}{U'_1} \ar@{^{(}->}[r]_{} \ar@{^{(}->}[u] &\dfrac{U_2}{U_1}\ar@{->>}[r]_{}\ar@{^{(}->}[u] & \dfrac{U''_2}{U''_1}\ar@{^{(}->}[u] \\
}
\end{equation}
whose rows and columns are admissible short exact sequences.

\vspace{0.1cm}

Let be $(\Delta', \delta')\in\Det_h(X')$ and $(\Delta'', \delta'')\in\Det_h(X'')$. We define a pair $(\Delta, \delta)$, where $\Delta$ is 
a function $\Gamma(X)\rar\Tors(G)$ defined as
$$
\Delta(U_1)\colon =\Delta'(U'_1)\otimes\Delta''(U''_1)
$$
and $\delta$ an arrow $\Delta(U_1)\otimes h\left(\dfrac{U_2}{U_1}\right)\rar\Delta(U_2)$, defined as the composition
$$
 \xymatrix{
 \Delta(U_1)\otimes h\left(\dfrac{U_2}{U_1}\right)=\ar[ddrr]_{\delta}\Delta'(U'_1)\otimes\Delta''(U''_1)\otimes h\left(\dfrac{U_2}{U_1}\right) \ar[rr]^-{1\otimes\lambda^{-1}}&& \Delta'(U'_1)\otimes\Delta''(U''_1)\otimes h\left(\dfrac{U'_2}{U'_1}\right)\otimes  h\left(\dfrac{U''_2}{U''_1}\right)\ar[d]^{1\otimes\sigma\otimes 1} \\
&& \Delta'(U'_1)\otimes h\left(\dfrac{U'_2}{U'_1}\right)\otimes \Delta''(U''_1)\otimes h\left(\dfrac{U''_2}{U''_1}\right)\ar[d]^{\delta'\otimes\delta''} \\
&& \Delta'(U'_2)\otimes\Delta''(U''_2)=\Delta(U_2)
}
$$
and we claim that $(\Delta, \delta)$ is an object if $\Det_h(X)$. This amounts to show that for this pair the diagram \eqref{delta} 
is commutative.

\vspace{0.1cm}

The proof consists in the construction of the diagram \eqref{delta} by tensorizing the analog diagrams for the determinantal theories 
$(\Delta', \delta')$ and $(\Delta'', \delta'')$, using the above given definitions of $\Delta$ and $\delta$. The resulting tensor product of the diagrams is equal to \eqref{delta}, provided that the diagram
$$
\xymatrix{  
 h\left(\dfrac{U'_2}{U'_1}\right)\otimes  h\left(\dfrac{U'_3}{U'_2}\right)\otimes  h\left(\dfrac{U''_2}{U''_1}\right)\otimes  h\left(\dfrac{U''_3}{U''_2}\right)  \ar[d]_{\lambda\otimes\lambda}\ar[rr]^{1\otimes\sigma\otimes 1}   &&  h\left(\dfrac{U'_2}{U'_1}\right)\otimes  h\left(\dfrac{U''_2}{U''_1}\right)\otimes  h\left(\dfrac{U'_3}{U'_2}\right)\otimes  h\left(\dfrac{U''_3}{U''_2}\right)\ar[d]^{\lambda\otimes\lambda} \\
 h\left(\dfrac{U'_3}{U'_1}\right)\otimes  h\left(\dfrac{U''_3}{U''_1}\right)  \ar[rd]_{\lambda}   &&  h\left(\dfrac{U_2}{U_1}\right)\otimes  h\left(\dfrac{U_3}{U_2}\right)\ar[ld]^{\lambda} \\
& h\left(\dfrac{U_3}{U_1}\right)  &
}
$$
commutes. But $(h, \lambda)$ is symmetric, hence this diagram commutes from Definition \eqref{newsymmetric}, applied to diagram 
\eqref{quotients}. Therefore $\mu$ is well defined on the objects. The proof that $\mu$ is a multiplicative equivalence in the sense of definition  \eqref{defmultgerbe} is straightforward. The theorem follows.
\end{proof}

Since such a determinantal theory $h$ can be interpreted as a multiplicative 
$G$-torsor of degree 1 on $\SA$, it determines a class in 
$H^2(S(\A), G)$. Then, from Theorem \eqref{multgerbe}, we obtain the 
following, which is the  analog of Corollary \eqref{cortorsmult}:

\begin{cor}
The class $[h]\in H^2(S(\A), G)$ gives rise to a cohomology class in 
$H^3(S(\limA), G)$.
\end{cor}
The corollary has an interpretation analog to that of corollary \eqref{cortorsmult}, as the ``second step delooping" of the cohomology of $S(\A)$ in terms of the cohomology of $S(\limA)$.
\section{Applications. Tate spaces and the iteration of the dimensional torsor}\label{tate} 
In this section we focus on the abelian category $\A=\vect$ of finite dimensional vector spaces over a field $k$.
\subsection{Tate spaces}
We have introduced the category $\T$ of Tate spaces in definition \eqref{tatespaces}. Let us denote by $\Ll_0$ the category of
linearly compact topological $k$-vector spaces and by $\Ll$ the category of locally linearly compact topological $k$-vector 
spaces and their morphisms, as introduced in \cite{l}, II.27.1 and II.27.9. We recall the following propositions, whose proofs can be found in \cite{pre}.
\begin{lm}
There are equivalences of categories: 
$$
\Phi_0:\Pro(\vect)\xrar{\sim}\Pro^s(\vect)\xrar{\sim}\Ll_0.
$$ 
In particular, the category $\Ll_0$is an abelian category.
\end{lm}
\begin{prop}\label{lefschetz}
There is an equivalence of categories: $\Phi:\T\xrar{\sim}\Ll$, whose restriction to the category $\Pro(\vect)$ is $\Phi_0$.
\end{prop}

As a consequence of Proposition \eqref{lefschetz}, $\Ll$ becomes endowed with a structure of an exact category, and it is self-dual 
(see Prop.\eqref{limduality}). 
\begin{prop}\label{lefschetz2}
(a) Under the identification of Proposition \eqref{lefschetz}, the class of admissible monomorphisms of $\Ll$ coincides with the class 
of its closed embeddings. \\
(b) Similarly, the class of admissible epimorphisms in $\Ll$ coincides with the class of continuous surjective morphisms $p:B\rar C$, 
such that the canonical bijection $\dfrac{B}{\ker(p)}\rar C$ is a homeomorphism.
\end{prop}
The above proposition allow us to identify $\T$ and $\Ll$. We also recall that 
the category $\T$ is not abelian. For example, the inclusion $k[t]\hrar k[[t]]$ is a non-admissible monomorphisms in $\T$. 
\begin{theorem}\label{tateqae}
The category $\T$ is partially abelian exact.
\end{theorem}
\begin{proof}
From the equivalence $\T\xrar{\sim}\Ll$ of Propositions \eqref{lefschetz} and \eqref{lefschetz2}, the closure of $\T$ under admissible intersections is clear, since the intersection of two closed subspaces of a space $X\in\T$ is closed. Thus $\T$ satisfies the (AIC). The dual condition (AIC)$^o$ comes from this fact because of the self-duality of $\T$.
\end{proof}
\subsection{Sato Grassmannians} The concept of Sato Grassmannian, introduced for any generalized Tate space $X$ in definition \eqref{G(X)}, coincides with the concept of semi-infinite Grassmannian in the case $X$ is a Tate vector 
space (i.e. when $\A=\vect$).
\begin{prop}
Let X be an object of $\T$. The Sato Grassmannian of X coincides with the set  $G(X)$ of open, linearly compact subspaces of X.
\end{prop}

\begin{proof}
(i) $\Gamma(X)\subset G(X)$. Let be $U\in\Gamma(X)$. By definition, 
$U\in\Pro^s(\vect)\xrar{\sim}\Ll_0$, so $U$ is a linearly compact subspace of 
$X$. Next, since $U$ is closed in $X$,  the projection 
$X\epi\dfrac{X}{U}$ is a continuous map. Since $\dfrac{X}{U}\in\Ind^s(\vect)\sim\allvect$, 
it follows that $\dfrac{X}{U}$ is a discrete space. Thus, $U=\pi^{-1}(0)$ is 
open. \\
(ii) $G(X)\subset\Gamma(X)$. Let be $V\in G(X)$ an open, linearly compact 
subspace of $X$. Since $V$ is linearly compact, $V\in\Pro^s(\vect)$. Also, 
$V$ is closed in $X$; from (27.8) of \cite{l}, the inclusion $V\hrar X$ is therefore a closed embedding, hence an admissible monomorphism. Being $V$ open, and a linear subspace of $X$, $V$ is a nuclear subspace. Thus, 
from (25.8)(c) of \cite{l}, the quotient $\dfrac{X}{V}$ is discrete, i.e. an object 
of $\Ind^s(\vect)$. It follows $V\in\Gamma(X)$, and we are done.
\end{proof}

\subsection{2-Tate spaces}
\begin{definition}
Let be $k$ a field. The category $\T_2=\dlim\T$ is called the category of {\it 2-Tate spaces} over $k$.
\end{definition}
The category $\T_2$ is thus in a natural way an exact category, and it is of course possible to further iterate the functor $\dlim$ 
and define, for all $n$, the exact category $=\T_n=\dlim\T_{n-1}=\dlim^{n}\vect$, of {\it n-Tate spaces}, but in this paper we shall 
be only concerned with 2-Tate spaces. We remark however that our definition of $n$-Tate spaces coincides with that of Arkhipov 
and Kremnizer, in \cite{ak}.
\subsection{Iteration}
Since from Theorem \eqref{tateqae} the category $\T$ is partially abelian exact, it is possible to extend the results on $\Dim$ and 
$\Det$ of the previous sections to the object of the category $\T_2$ of 2-Tate spaces.

\vspace{0.1cm} 

Let be $\Xi\in\T_2$ a 2-Tate space. We shall denote, from now on, by $\Dim^{(1)}$ the universal dimensional $\Zee$-torsor 
$\Dim$ over the category $\T$, constructed in section \eqref{sectionunivdim}.  As we have seen, the collection of $\Dim^{(1)}(V)$, for $V\in\Ob \T$ 
forms a symmetric determinantal theory. 

\begin{theorem}
It is possible to define a 
{\it ($\Zee$-tors)-torsor} (i.e. a $\Zee$-gerbe), associated to the object $\Xi\in\T_2$, as 
$$
\Dim^{(2)}(\Xi):=\Det_{\Dim^{(1)}}(\Xi).
$$
\end{theorem}
The gerbe $\Dim^{(2)}$ is multiplicative with respect to admissible short exact sequences of  $\T_2$.

\vspace{0.1cm}

It is also possible to define $\Det^{(2)}(\Xi)$, the universal determinantal gerbe of $\Xi$, over the universal determinantal theory 
on $\T$. It results a multiplicative 2-gerbe over $k^*$. \\
This theory  coincides with the theory of gerbel theories and 2-gerbes contained in \cite{ak}. We postpone to a forthcoming paper a 
more detailed proof of this equivalence.

{\small D\'epartement de Math\'ematiques\\
Universit\'e Paris 13\\
93430 Villetaneuse\\
France\\
\tt previdi@math.univ-paris13.fr}

\end{document}